\documentclass[a4paper,11pt]{article}
\usepackage[T1]{fontenc}
\usepackage{lmodern}
\usepackage{amsmath,amssymb,amsthm}
\usepackage{hyperref}
\usepackage[capitalise,nameinlink,noabbrev]{cleveref}
\crefname{assumption}{Assumption}{Assumptions}
\crefname{conjecture}{Conjecture}{Conjectures}
\crefname{lemma}{Lemma}{Lemmas}
\crefname{proposition}{Proposition}{Propositions}
\crefname{equation}{}{}
\crefname{algocf}{Algorithm}{Algorithms}
\crefname{Question}{Question}{Questions}
\newtheorem{example}{Example}
\newtheorem{remark}{Remark}
\newtheorem{question}{Question}
\newtheorem{assumption}{Assumption}
\newtheorem{theorem}{Theorem}
\newtheorem{lemma}{Lemma}
\newtheorem{definition}{Definition}
\usepackage[a4paper, margin=3cm]{geometry}

\usepackage{graphicx}
\usepackage{amsfonts,cite,float,url,booktabs,xcolor,mathrsfs}
\usepackage{subcaption}
\usepackage[section]{placeins}
\usepackage[algo2e,ruled,linesnumbered]{algorithm2e}
\newcommand{\argmin}{\mathop{\arg\min}}

\newcommand{\Ttran}{\mathsf{T}}
\newcommand{\Htran}{\mathsf{H}}
\newcommand{\de}{\,\mathrm{d}}
\newcommand{\mi}{\!\cdot\!\mathsf{i}}
\newcommand{\Rq}{\mathsf{Rf}}
\newcommand{\defi}{:=}
\newcommand{\order}{\mathcal{O}}
\newcommand{\R}{\mathbb{R}}
\newcommand{\C}{\mathbb{C}}
\newcommand{\rk}{\mathrm{rank}}
\newcommand{\region}{\mathcal{D}}
\newcommand{\range}{\mathsf{range}}
\newcommand{\fig}{eps}
\newcommand{\Crval}{C_{\mathsf{val}}}
\newcommand{\Crvec}{C_{\mathsf{vec}}}
\newcommand{\kval}{\kappa_{\mathsf{val}}}
\newcommand{\kvec}{\kappa_{\mathsf{vec}}}
\newcommand{\figsizeD}{0.4\textwidth}
\providecommand{\spa}[1]{\mathsf{span}\{#1\}}
\providecommand{\abs}[1]{\lvert#1\rvert}
\providecommand{\norm}[1]{\lVert#1\rVert}
\providecommand{\bigabs}[1]{\bigl\lvert#1\bigr\rvert}
\providecommand{\Bigabs}[1]{\Bigl\lvert#1\Bigr\rvert}
\providecommand{\bignorm}[1]{\bigl\lVert#1\bigr\rVert}
\providecommand{\Bignorm}[1]{\Bigl\lVert#1\Bigr\rVert}

\begin{document}
\title{Stabilizing the Rayleigh--Ritz procedure by randomization}
\author{
    Nian Shao\thanks{
         Institute of Mathematics, EPF Lausanne, 1015 Lausanne, Switzerland
         (\href{mailto:nian.shao@epfl.ch}{nian.shao@epfl.ch}).
     }
}
\maketitle

\begin{abstract}
Extracting approximate eigenpairs from a prescribed subspace is of fundamental importance in eigenvalue computation.
Although the standard Rayleigh--Ritz procedure is widely used for this purpose, it may suffer from deteriorated convergence of Ritz values and may even fail to produce convergent Ritz vectors, unless the compressed eigenvalue problem is assumed to be well-conditioned.
In this paper, we address the eigenpair extraction problem by introducing a randomized Rayleigh--Ritz procedure.
Our analysis shows that, for a simple eigenpair, the eigenvalue and eigenvector condition numbers of the randomly compressed generalized eigenvalue problem are, with high probability, comparable to those of the original problem, which yields a reliable approximate eigenpair from the extraction.
The result also extends naturally to nonlinear eigenvalue problems.
\end{abstract}



\section{Introduction}
This paper is concerned with the following eigenpair extraction problem.
\begin{question}
    \label{Q}
    Let $(\lambda,v)$ be a simple eigenpair of $A_{0}\in\C^{n\times n}$. Given an $m$-dimensional subspace $\mathcal{W}_{\epsilon}$, where $1<m\ll n$ and 
\begin{equation*}
    \epsilon\defi \angle(v,\mathcal{W}_{\epsilon})
\end{equation*}
denotes the canonical angle between $v$ and $\mathcal{W}_{\epsilon}$, how can we reliably extract a good approximation of the eigenpair $(\lambda,v)$ from the prescribed subspace $\mathcal{W}_{\epsilon}$? 
\end{question}

\cref{Q} is of fundamental importance in eigenvalue computation, since such subspaces may arise from a variety of methods, such as Krylov subspace methods and reduced-basis techniques.
Unlike solving linear systems, where the optimal solution can be computed directly via a least squares formulation, extracting an eigenpair from a given subspace is considerably more challenging.
While the (standard) Rayleigh--Ritz procedure \cite[Chap.~4.4]{Stewart2001} is widely used for this purpose, Stewart \cite[p.~282]{Stewart2001} observes that ``we have no guarantee that the result approximates the desired eigenpair, \ldots, this difficulty has no easy answer'', highlighting \cref{Q} as a long-standing challenge.
Indeed, unless $A_{0}$ is Hermitian and $\lambda$ is an exterior eigenvalue, several stability issues may arise in the standard Rayleigh--Ritz procedure: the compressed eigenvalue problem can be arbitrarily more ill-conditioned than the original one, which leads to deterioration in the convergence of Ritz values and possible failure of convergence of Ritz vectors; see \Cref{sec:stab} for details.

In this paper, we introduce a \emph{randomized Rayleigh--Ritz} (RRR) procedure (\cref{alg:RRR-linear})  that addresses \cref{Q}.
Before describing the RRR procedure, we briefly recall the standard Rayleigh--Ritz procedure. For the eigenvalue problem associated with $A_{0}$, it computes \emph{Ritz pairs} $(\mu_{\epsilon}, W_{\epsilon} y_{\epsilon})$  by solving a compressed eigenvalue problem
\begin{equation*}
    W_{\epsilon}^{\Htran}A_{0}W_{\epsilon}y_{\epsilon}=\mu_{\epsilon} y_{\epsilon},
\end{equation*}
where $W_{\epsilon}$ is an orthonormal basis of $\mathcal{W}_{\epsilon}$.
In the RRR procedure, instead of imposing a Galerkin condition, we impose a Petrov--Galerkin condition using a complex Gaussian random matrix 
$\Omega\in\C^{n\times m}$ and solve the following compressed generalized eigenvalue problem:
\begin{equation*}
    \Omega^{\Htran}A_{0}W_{\epsilon}y_{\epsilon}=\mu_{\epsilon}\Omega^{\Htran}W_{\epsilon} y_{\epsilon}.
\end{equation*}
To distinguish these quantities from standard Ritz values and vectors, we refer to $\mu_{\epsilon}$ and $W_{\epsilon}y_{\epsilon}$ satisfying this compressed generalized eigenvalue problem as \emph{randomized Ritz values} and \emph{randomized Ritz vectors}, respectively. We then refine randomized Ritz values using the usual Rayleigh quotient
\begin{equation*}
    \rho_{\epsilon} = \frac{(W_{\epsilon}y_{\epsilon})^{\Htran}A_{0}(W_{\epsilon}y_{\epsilon})}{(W_{\epsilon}y_{\epsilon})^{\Htran}(W_{\epsilon}y_{\epsilon})}.
\end{equation*}
Our theoretical analysis shows that, for a \emph{simple} eigenpair $(\lambda,v)$ of $A_{0}$, with high probability, the RRR procedure returns a randomized Ritz pair $(\mu_{\epsilon},W_{\epsilon}y_{\epsilon})$ and a corresponding refined randomized Ritz value $\rho_{\epsilon}$ such that
\begin{equation*}
    \abs{\lambda-\rho_{\epsilon}} = \order(\epsilon^{\gamma}),\quad 
    \abs{\lambda-\mu_{\epsilon}} = \order(\epsilon)
    \quad \text{and} \quad
    \angle(v, W_{\epsilon}y_{\epsilon}) = \order(\epsilon),
\end{equation*}
where $\gamma=2$ for a Hermitian $A_{0}$ and $\gamma=1$ otherwise.
We emphasize that the $\order(\cdot)$ notation is employed solely for notational convenience; all probabilistic bounds in this paper are non-asymptotic.
In particular, if there exists a family of subspaces $\mathcal{W}_{\epsilon}$ satisfying $\epsilon\to 0$, then up to a constant independent of the subspaces $\mathcal{W}_{\epsilon}$, the convergence rates of randomized Ritz vectors and their corresponding refined randomized Ritz values match those of the ideal projections $W_{\epsilon}W_{\epsilon}^{\Htran}v$ and their Rayleigh quotients, respectively.
This is achieved by establishing the stability\footnote{Here, ``stability'' does not refer to rounding errors. Instead, it refers to the property that ensures that consistency, $\angle(v,W_{\epsilon})\to 0$, implies convergence of the extracted eigenpair, e.g., $\angle(v,W_{\epsilon}y_{\epsilon})\to 0$; see \cite{Hackbusch2014} for a detailed discussion of the notion of ``stability'' in numerical mathematics.} of the RRR procedure, that is, the random compression preserves eigenvalue and eigenvector condition numbers up to a modest factor with high probability; see \cref{lem:kappa} for details.
To the best of our knowledge, no existing extraction method achieves this without additional assumptions on the compressed eigenvalue problem.

\begin{algorithm2e}[H]
    \caption{RRR procedure for standard eigenvalue problems}
    \label{alg:RRR-linear}
    \KwIn{$A_{0}\in\C^{n\times n}$ and an orthonormal basis $W_{\epsilon}\in\C^{n\times m}$\;}
    Draw a complex Gaussian random matrix $\Omega\in\C^{n\times m}$\; 
    Solve the compressed generalized eigenvalue problem of the pencil $\Omega^{\Htran}A_{0}W_{\epsilon}-\xi \Omega^{\Htran}W_{\epsilon}$ and select wanted randomized Ritz pairs $(\mu_{\epsilon},W_{\epsilon}y_{\epsilon})$\;
    Refine randomized Ritz values by $\rho_{\epsilon} = \frac{(W_{\epsilon}y_{\epsilon})^{\Htran}A_{0}(W_{\epsilon}y_{\epsilon})}{(W_{\epsilon}y_{\epsilon})^{\Htran}(W_{\epsilon}y_{\epsilon})}$\;
    \Return{Randomized Ritz pairs $(\mu_{\epsilon},W_{\epsilon}y_{\epsilon})$ and refined randomized Ritz values $\rho_{\epsilon}$\;}
\end{algorithm2e}

We remark that \cref{alg:RRR-linear} is closely related to the sketched Rayleigh--Ritz procedure proposed in \cite[Sec.~6]{Nakatsukasa2024}. However, there are several differences between the two methods. In the sketched Rayleigh--Ritz procedure, a random matrix is used as an oblivious subspace embedding, and oversampling is required, whereas the RRR procedure does not require oversampling. Moreover, the sketched Rayleigh--Ritz procedure provides only \emph{a posteriori} estimates, while we establish \emph{a priori} convergence results.
Other related algorithms for extracting eigenpairs include the refined \cite{Jia1997} and harmonic \cite{Morgan1991} Rayleigh--Ritz procedures. Compared with the standard Rayleigh--Ritz procedure, these variants require additional computations, such as the SVD or QR decomposition of $n\times m$ matrices. From a theoretical perspective, they can be interpreted as performing an alternating optimization over $\mu$ and $y$, and they work well empirically for standard eigenvalue problems. However, as we will show in \Cref{sec:stab}, both Ritz values and (refined and harmonic) Ritz vectors may fail to converge for generalized eigenvalue problems. Consequently, existing convergence guarantees for the refined and harmonic Rayleigh--Ritz procedures require additional assumptions on the compressed eigenvalue problems, and these assumptions may fail even if the original problems are well-conditioned.

The remainder of this paper is organized as follows.
In \Cref{sec:stab}, we present several examples illustrating the stability issues of the standard Rayleigh--Ritz procedure.
In \Cref{sec:algo}, we formally describe the RRR procedure for nonlinear eigenvalue problems.
\Cref{sec:theory} contains our main theoretical results, including the stability of the RRR procedure (\cref{lem:kappa}) and convergence guarantees for its output (\cref{thm:ritzpair,thm:Rf,thm:nls}).
Numerical experiments are presented in \Cref{sec:numexp}.

\begin{remark}
    As in the standard Rayleigh--Ritz procedure, the RRR procedure returns several approximate eigenpairs, and the wanted ones are selected according to a user-specified criterion; for example, MATLAB's \texttt{eigs} returns by default the approximate eigenvalues of largest magnitude. Our analysis shows that a reliable approximate eigenpair is among them; the selection itself is independent of the extraction and is not discussed further.
\end{remark}

\paragraph{Notation}
Throughout this paper, $\norm{\cdot}$ denotes the Euclidean norm for vectors and the spectral norm for matrices.
We use $e_{i}$ to denote the $i$-th canonical basis vector.
A matrix $\Omega$ is called a complex Gaussian random matrix if the real and imaginary parts of its entries are independent and identically distributed according to $\mathcal{N}(0,1/2)$.
The canonical angle between $\spa{v}$ and $\range(W)$ is denoted by $\angle(v,W)$ or $\angle(v,\range(W))$.
For a subspace $\mathcal{W}\subset\C^{n}$, we denote its orthogonal complement by $\mathcal{W}^{\perp}$, so that $\mathcal{W}\oplus \mathcal{W}^{\perp} = \C^{n}$.
A matrix $Q\in\C^{n\times m}$ is called orthonormal if its columns form an orthonormal basis, that is, $Q^{\Htran}Q = I_{m}$, where $I_{m}$ is the $m\times m$ identity matrix.
For a matrix-valued function $A(\xi)$ and a nonempty bounded domain $\region$, we define $
\norm{A}_{\region} := \sup_{\xi\in\region}\norm{A(\xi)}$.

\section{Stability issues of the standard Rayleigh--Ritz procedure}
\label{sec:stab}
In this section, we present several examples illustrating the stability issues of the standard Rayleigh--Ritz procedure. Throughout this section, as well as in all results of this paper, we assume that $(\lambda, v)$ is a simple eigenpair.

We begin with a Hermitian matrix $A_{0}$. When $\lambda$ is the largest (or smallest) eigenvalue, the minimax principle ensures that the second eigenvalue of the compressed matrix $W_{\epsilon}^{\Htran} A_{0} W_{\epsilon}$ is bounded by the second eigenvalue of $A_{0}$. If $\mu_{\epsilon}$ converges to $\lambda$, then it remains well separated from the remaining eigenvalues of $W_{\epsilon}^{\Htran} A_{0} W_{\epsilon}$. Consequently, Ritz values converge quadratically, while Ritz vectors converge linearly; see, for instance, \cite[Sec.~4.3.2]{Saad2011}.

In contrast, for an interior eigenvalue of a Hermitian matrix, the convergence of Ritz values may deteriorate to linear, and Ritz vectors may even fail to converge. This occurs because the target Ritz value can become arbitrarily close to other eigenvalues of $W_{\epsilon}^{\Htran} A_{0} W_{\epsilon}$.
\begin{example}[Interior eigenvalue of a Hermitian matrix: linear convergence of Ritz values and failure of convergence of Ritz vectors]
    \label{eg:H}
    Consider the eigenvalue $0$ of the matrix  
    \begin{equation*}
        A_{0} = \begin{bmatrix}
            -1 & 0 & 0 \\ 
            0 & 0 & 0\\ 
            0 & 0 & 1
        \end{bmatrix}
        \quad\text{and}\quad 
        W_{\epsilon} = \begin{bmatrix}
            \epsilon/\sqrt{2} & 1/\sqrt{2}\\ 
            \sqrt{1-\epsilon^{2}} & 0 \\ 
            \epsilon/\sqrt{2} & -1/\sqrt{2}
        \end{bmatrix}.
    \end{equation*}
    Then the compressed matrix is 
    \begin{equation*}
        W_{\epsilon}^{\Htran}A_{0}W_{\epsilon} = \begin{bmatrix}
            0 & -\epsilon\\ 
            -\epsilon & 0
        \end{bmatrix}.
    \end{equation*}
    Ritz values are $\pm\epsilon$, yielding only linear convergence.
    The corresponding Ritz vectors are $W_{\epsilon}[1,-1]^{\Ttran}$ and $W_{\epsilon}[1,1]^{\Ttran}$, which provide no accuracy for the target eigenvector.
\end{example}

For a non-Hermitian matrix $A_{0}$, as shown in \cite{Jia2001}, applying Elsner's theorem 
to the compressed matrix ensures the convergence of Ritz values.  However, the convergence rate may deteriorate substantially to $\order(\epsilon^{1/m})$, since the compressed matrix can become (nearly) defective.
\begin{example}[Non-Hermitian matrix: $\order(\epsilon^{1/m})$ convergence of Ritz values]
    \label{eg:nH}
    Consider the eigenvalue $0$ of the matrix 
\begin{equation*}
    A_{0} = \begin{bmatrix}
        0 & 1 & 0\\ 
        0 & 0 & I_{m-1}\\ 
        0 & 1 & 0
    \end{bmatrix}
    \quad\text{and}\quad 
    W_{\epsilon}=\begin{bmatrix}
        \sqrt{1-\epsilon^{2}} & 0\\ 
        0 & I_{m-1} \\ 
        \epsilon & 0
    \end{bmatrix}.
\end{equation*}
The compressed matrix is 
\begin{equation*}
    W_{\epsilon}^{\Htran}A_{0}W_{\epsilon}
     = 		\begin{bmatrix}
			0        & \sqrt{1-\epsilon^2} + \epsilon & 0       \\
			0        & 0                              & I_{m-2} \\
			\epsilon & 0                              & 0
		\end{bmatrix},
\end{equation*}
whose eigenvalues are approximately of magnitude $\order(\epsilon^{1/m})$ as $\epsilon\to 0$.
\end{example}

An even more severe situation arises for generalized eigenvalue problems, where Ritz values may fail to converge. This occurs because the compressed matrix pencil can be $\order(\epsilon)$-close to singular even if the original pencil is regular. Consequently, the convergence result for standard eigenvalue problems cannot be extended to the generalized case. 

\begin{example}[Generalized eigenvalue problem: Ritz values, refined Ritz vectors and harmonic Ritz pairs fail to converge\footnote{This example indicates that the convergence results for Ritz values and refined Ritz vectors in \cite{Jia2025} require additional assumptions on the compressed eigenvalue problem, which are not always satisfied.}]
    \label{eg:G}
    Consider the matrix pencil $A_{0}-\xi A_{1}$, where
\begin{equation*}
    A_{0}= \begin{bmatrix}
        0 & 0 & 1\\
        0 & 2 & 0\\
        0 & 3 & 0
    \end{bmatrix}
    \quad\text{and}\quad
    A_{1}= \begin{bmatrix}
        0 & 0 & 1\\
        1 & 0 & 0\\
        0 & 1 & 0
    \end{bmatrix}.
\end{equation*}
Since $\det (A_{0}-\xi A_{1}) = -\xi(1-\xi)(3-\xi)$, the pencil is regular with three simple eigenvalues $0$, $1$ and $3$, associated with eigenvectors $[1,0,0]^{\Ttran}$, $[0,0,1]^{\Ttran}$ and $[2,3,0]^{\Ttran}$, respectively.
Consider the following trial subspace and its corresponding compressed pencil
\begin{equation*}
    W_{\epsilon} = \begin{bmatrix}
        \sqrt{1-\epsilon^{2}} & 0\\
        0 & 1\\
        \epsilon & 0
    \end{bmatrix}
    \quad\text{and}\quad
    W_{\epsilon}^{\Htran}(A_{0}-\xi A_{1})W_{\epsilon} = \begin{bmatrix}
        \epsilon &0\\
        0& 1
    \end{bmatrix}
    \begin{bmatrix}
        1-\xi & 3-\xi\\
        -\xi & 2
    \end{bmatrix}\begin{bmatrix}
        \sqrt{1-\epsilon^{2}} & 0\\
        0& 1
    \end{bmatrix}.
\end{equation*}
Ritz values are $2$ and $-1$ for every $\epsilon>0$, which fail to converge to an eigenvalue of $A_{0}-\xi A_{1}$ as $\epsilon \to 0$.

Refined Ritz vectors \cite{Jia1997}, defined as minimizers of $\norm{(A_{0}-\mu A_{1})w}$ over unit vectors $w\in\range(W_{\epsilon})$ for a Ritz value $\mu$, also fail to converge. As $\epsilon\to 0$, the refined Ritz vectors associated with $\mu=2$ and $\mu=-1$ converge to the directions of $[8,\sqrt{65}-1,0]^{\Ttran}$ and $[4,19-\sqrt{377},0]^{\Ttran}$, respectively, rather than the target eigenvector $[1,0,0]^{\Ttran}$.

Harmonic Ritz pairs \cite{Morgan1991} also fail to converge, even if the shift is chosen as the target eigenvalue $0$. They are defined by the condition $(A_{0}-\theta A_{1})W_{\epsilon}y\perp A_{0}W_{\epsilon}$, which leads to
\begin{equation*}
    \det\bigl((A_{0}W_{\epsilon})^{\Htran}(A_{0}-\theta A_{1})W_{\epsilon}\bigr)
    =\det\begin{bmatrix}
        \epsilon^{2}(1-\theta) & 0\\
        -2\theta\sqrt{1-\epsilon^{2}} & 13-3\theta
    \end{bmatrix}
    =\epsilon^{2}(1-\theta)(13-3\theta).
\end{equation*}
Hence, harmonic Ritz values are $1$ and $13/3$ for every $\epsilon>0$, and the harmonic Ritz vectors associated with $\theta=1$ and $\theta=13/3$ converge to the directions of $[5,1,0]^{\Ttran}$ and $[0,1,0]^{\Ttran}$, respectively.
\end{example}

We summarize the instabilities of the standard Rayleigh--Ritz procedure and compare them with those of the RRR procedure in \cref{tab:cmp}.
The corresponding convergence results for the latter will be established in \cref{thm:nls,thm:Rf,thm:ritzpair}.

\begin{table}[htbp]
    \centering
    \caption{Convergence guarantees of the standard and randomized Rayleigh--Ritz procedures with $\epsilon = \angle(v, W_{\epsilon})$. 
A cross `$\times$' indicates that the convergence may fail, while the question mark `?' denotes that the result is currently unclear. 
An eigenvalue problem $A(\cdot)$ is called Hermitian (at $\lambda$) if the matrix $A(\lambda)$ is Hermitian for the target eigenvalue $\lambda$. 
All convergence statements for the RRR procedure are probabilistic and non-asymptotic.}
    \begin{tabular}{lccccc}
        \toprule
        & \multicolumn{2}{c}{Rayleigh--Ritz}  & \multicolumn{3}{c}{randomized Rayleigh--Ritz} \\ \cmidrule(lr){2-3} \cmidrule(lr){4-6}
        & value & vector  & refined value & value & vector\\ \midrule
        Linear Hermitian (exterior) & $\order(\epsilon^{2})$ & $\order(\epsilon)$ & $\order(\epsilon^{2})$  & $\order(\epsilon)$ & $\order(\epsilon)$\\ 
        Linear Hermitian (interior) & $\order(\epsilon)$ & $\times$ & $\order(\epsilon^{2})$  & $\order(\epsilon)$ & $\order(\epsilon)$\\ 
        Nonlinear Hermitian& ? & $\times$ & $\order(\epsilon^{2})$  & $\order(\epsilon)$ & $\order(\epsilon)$\\ 
        Standard non-Hermitian & $\order(\epsilon^{1/m})$ & $\times$ & $\order(\epsilon)$  & $\order(\epsilon)$ & $\order(\epsilon)$\\ 
        Generalized non-Hermitian & $\times$ & $\times$ & $\order(\epsilon)$  & $\order(\epsilon)$ & $\order(\epsilon)$\\ 
        Nonlinear non-Hermitian & $\times$ & $\times$ & $\order(\epsilon)$  & $\order(\epsilon)$ & $\order(\epsilon)$\\
         \bottomrule
        \end{tabular}
    \label{tab:cmp}
\end{table}

\begin{remark}
For a non-Hermitian problem with left and right eigenvectors $u$ and $v$, suppose that we are given two subspaces $\mathcal{W}_{\varepsilon,L}$ and $\mathcal{W}_{\epsilon,R}$ such that $\varepsilon=\angle(u,\mathcal{W}_{\varepsilon,L})$ and $\epsilon=\angle(v,\mathcal{W}_{\epsilon,R})$.
We can apply the RRR procedure separately to obtain left and right randomized Ritz vectors $W_{\varepsilon,L}y_{\varepsilon,L}$ and $W_{\epsilon,R}y_{\epsilon,R}$.
With an additional refinement analogous to the two-sided Rayleigh quotient \cite{Schwetlick2012}, the accuracy of refined randomized Ritz values can be improved to $\order(\varepsilon\epsilon)$.
\end{remark}

We conclude this section by commenting on the instabilities of the standard Rayleigh--Ritz procedure and giving an intuitive explanation of how randomization can stabilize it.
Consider a matrix $A_{0}$ with a simple eigenpair $(\lambda,v)$. Let $W$ be an orthonormal matrix such that $v\in\range(W)$.
The simplicity of $\lambda$ implies that zero is a simple singular value of $(A_{0}-\lambda I)W$.
However, the compressed matrix $W^{\Htran}(A_{0}-\lambda I)W$ does \emph{not} generally inherit this property unless $A_{0}$ is Hermitian and $\lambda$ is an exterior eigenvalue.
In fact, the compression may degenerate to a zero matrix in the Hermitian case (see \cref{eg:H}) or form a Jordan block in the non-Hermitian case (see \cref{eg:nH}).
As a result, the convergence of Ritz values may deteriorate and Ritz vectors may fail to converge.
In contrast, compressing $(A_{0}-\lambda I)W$ with an oblivious (complex) Gaussian random matrix preserves the simplicity of the zero singular value almost surely.
Specifically, in \cref{lem:kappa}, we show that, with high probability, the eigenvalue and eigenvector condition numbers of the compressed problem (at $\lambda$) are bounded by moderate multiples of those of the original problem.
These properties enable us to establish a probabilistic convergence result for the RRR procedure.

\section{RRR procedure for nonlinear eigenvalue problems}
\label{sec:algo}
In this section, we extend the RRR procedure to nonlinear eigenvalue problems (\cref{alg:RRR}).
Readers interested solely in standard eigenvalue problems may skip this section, with the exception of the perturbation result in \cref{lem:pert}.

\begin{assumption}
    \label{asp:A}
    Throughout this paper, we assume that $A(\cdot)$ is a regular holomorphic matrix-valued function defined on a nonempty bounded domain $\region\subset\C$ that admits an analytic extension to $\partial\region$.
    Here, regular means that $\det A(\xi_{0})\neq 0$ for some $\xi_{0}\in\region$.
\end{assumption}

The nonlinear eigenvalue problem \cite{Guttel2017} consists of finding eigenpairs $(\lambda,v)$ such that
\begin{equation*}
    A(\lambda) v = 0,\quad \lambda\in\region\quad\text{and}\quad v\in\C^{n}\setminus\{0\}.
\end{equation*}
In particular, if the null space of the matrix $A(\lambda)$ is one-dimensional and $u^{\Htran}A^{\prime}(\lambda)v\neq 0$ for a left eigenvector $u$, then the eigenvalue $\lambda$ of $A(\cdot)$ is called simple \cite[Prop.~1]{Neumaier1985}.

As for standard eigenvalue problems, we impose a Petrov--Galerkin condition on $A(\cdot)W_{\epsilon}$ using a complex Gaussian random matrix $\Omega\in\C^{n\times m}$ and solve the following compressed nonlinear eigenvalue problem:
\begin{equation}
    \label{eq:RRR}
    B_{\epsilon}(\mu_{\epsilon})y_{\epsilon} = 0,\quad 
    \mu_{\epsilon}\in\region,\quad 
    y_{\epsilon}\in\C^{m}\setminus\{0\},\quad\text{where}\quad 
    B_{\epsilon}(\xi) \defi \Omega^{\Htran}A(\xi)W_{\epsilon}.
\end{equation}

To refine randomized Ritz values, we may use the \emph{Rayleigh functional} \cite{Schwetlick2012}.
\begin{definition}
    \label{def:Rf}
    Let $\lambda$ be a simple eigenvalue of $A(\cdot)$ associated with an eigenvector $v$. Given two open sets $\region_{\lambda}\subset \region$ and $\region_{v}\subset \C^{n}$ such that $\lambda\in\region_{\lambda}$ and $v\in\region_{v}$, a functional $\Rq\colon \region_{v}\mapsto \region_{\lambda}$ is called a Rayleigh functional of $A(\cdot)$ if
    \begin{equation*}
        \Rq(\alpha w)=\Rq(w),\quad 
        w^{\Htran}A\bigl(\Rq(w)\bigr)w= 0,\quad 
        w^{\Htran}A^{\prime}\bigl(\Rq(w)\bigr)w\neq 0
    \end{equation*}
    hold for all $ \alpha\in\C\setminus\{0\}$ and $w\in\region_{v}$.
\end{definition}

For a simple eigenpair $(\lambda,v)$, if $v^{\Htran}A^{\prime}(\lambda)v\neq 0$, then the Rayleigh functional exists uniquely \cite[Thm.~21]{Schwetlick2012}. In particular, if $A(\lambda)$ is Hermitian, the simplicity of $\lambda$ implies the existence of $\Rq$. For a linear eigenvalue problem $A(\xi)=A_{0}-\xi A_{1}$, if $v^{\Htran}A_{1}v\neq 0$, then the Rayleigh functional is the usual Rayleigh quotient
\begin{equation*}
    \Rq(w) = \frac{w^{\Htran}A_{0}w}{w^{\Htran}A_{1}w}.
\end{equation*}

If a Rayleigh functional is available,  randomized Ritz values can be refined by $\rho_{\epsilon}=\Rq(W_{\epsilon}y_{\epsilon})$, where $W_{\epsilon}y_{\epsilon}$ is a randomized Ritz vector.  
However, for non-Hermitian problems, a Rayleigh functional may not exist because $v^{\Htran}A^{\prime}(\lambda)v$ could be zero, as is the case for the matrix pencil in \cref{eg:G}. In this situation, we can still perform the refinement by finding a stationary point 
\begin{equation}
    \label{eq:nls}
    \rho_{\epsilon} \in \Bigl\{ \rho\in\C \Bigm\vert \frac{\de }{\de \bar{\rho}} \norm{A(\rho)W_{\epsilon}y_{\epsilon}}^{2} = 0  \Bigr\},
\end{equation}
where $\bar{\rho}$ is the complex conjugate of $\rho$ and 
\begin{equation*}
    \frac{\de }{\de \bar{\rho}} \norm{A(\rho)W_{\epsilon}y_{\epsilon}}^{2} = \bigl(A^{\prime}(\rho)W_{\epsilon}y_{\epsilon}\bigr)^{\Htran}\bigl(A(\rho)W_{\epsilon}y_{\epsilon}\bigr)
\end{equation*}
is the Wirtinger derivative.
When $A(\xi)=A_{0}-\xi A_{1}$, this reduces to 
\begin{equation*}
    \rho_{\epsilon} = \frac{(W_{\epsilon}y_{\epsilon})^{\Htran}A_{1}^{\Htran}A_{0}(W_{\epsilon}y_{\epsilon})}{(W_{\epsilon}y_{\epsilon})^{\Htran}A_{1}^{\Htran}A_{1}(W_{\epsilon}y_{\epsilon})}.
\end{equation*}
For a general $A(\cdot)$, finding a stationary point of \cref{eq:nls} is closely related to a nonlinear least squares problem, which can be solved by the Gauss--Newton method or the Levenberg--Marquardt method \cite[Chap.~10]{Nocedal2006}.
If the set in \cref{eq:nls} contains multiple stationary points, additional selection criteria are required to identify an appropriate $\rho_{\epsilon}$, for instance, choosing the one closest to a prescribed shift.

We remark that the refinements described above, including both Rayleigh functional and stationary point approaches, typically do \emph{not} require additional matrix-vector multiplications with $A(\cdot)$, since $A(\cdot)W_{\epsilon}$ has already been computed when forming the compressed eigenvalue problem.

\begin{algorithm2e}[H]
    \caption{RRR procedure for nonlinear eigenvalue problems}
    \label{alg:RRR}
    \KwIn{$A(\cdot)\in\C^{n\times n}$ and an orthonormal basis $W_{\epsilon}\in\C^{n\times m}$\;}
    Draw a complex Gaussian random matrix $\Omega\in\C^{n\times m}$\; 
    Solve the compressed nonlinear eigenvalue problem of $B_{\epsilon}(\xi)=\Omega^{\Htran}A(\xi)W_{\epsilon}$ and select wanted randomized Ritz pairs $(\mu_{\epsilon}, W_{\epsilon}y_{\epsilon})$\;
    Refine randomized Ritz values to obtain $\rho_{\epsilon}$ by evaluating the Rayleigh functional $\Rq(W_{\epsilon}y_{\epsilon})$ or finding a suitable stationary point of \cref{eq:nls}\;
    \Return{Randomized Ritz pairs $(\mu_{\epsilon},W_{\epsilon}y_{\epsilon})$ and refined randomized Ritz values $\rho_{\epsilon}$\;}
\end{algorithm2e}

We conclude this section with the following first-order perturbation result.
\begin{theorem}
    \label{lem:pert}
    Given a matrix-valued function $B(\cdot)\colon \region\mapsto\C^{m\times m}$ satisfying \cref{asp:A}, let $\lambda$ be a simple eigenvalue of $B(\cdot)$ associated with left and right eigenvectors $x$ and $y$.
    Let $Z_{\perp}$ and $Y_{\perp}$ be orthonormal bases of $\spa{B^{\prime}(\lambda)y}^{\perp}$ and $\spa{y}^{\perp}$, respectively. Then $Z_{\perp}^{\Htran}B(\lambda)Y_{\perp}$ is nonsingular.

    Consider a perturbation $B+\Delta B$, where $\Delta B(\cdot)\colon \region\mapsto \C^{m\times m}$ is a holomorphic matrix-valued function that admits an analytic extension to $\partial\region$.
    Let 
    \begin{equation*}
        M = \max\bigl\{ \norm{B}_{\region},\norm{B^{\prime}}_{\region},\norm{B^{\prime\prime}}_{\region}\bigr\}
        \quad\text{and}\quad 
        \epsilon = \max\bigl\{\norm{\Delta B}_{\region},\norm{\Delta B^{\prime}}_{\region}, \norm{\Delta B^{\prime\prime}}_{\region}\bigr\}.
    \end{equation*}
    If
    \begin{equation}
        \label{eq:pertRadius}
        16M\kval(B,\lambda)\bigl(\kval(B,\lambda)+\kvec(B,\lambda)\bigr)\epsilon\leq 1
        \quad\text{and}\quad
        2\kval(B,\lambda)\,\epsilon<\mathrm{dist}(\lambda,\partial\region),
    \end{equation}
    then there exists an eigenpair $(\lambda +\Delta \lambda,y+\Delta y)$ of $B+\Delta B$ such that 
    \begin{equation*}
        \abs{\Delta \lambda} \leq \kval(B,\lambda)\cdot\norm{\Delta B(\lambda)} +C \epsilon^{2} 
        \quad\text{and}\quad  
        \frac{\norm{\Delta y}}{\norm{y}} \leq \kvec(B,\lambda)\cdot\norm{\Delta B(\lambda)}+C\epsilon^{2},
    \end{equation*}
    where
    \begin{equation*}
        \kval(B,\lambda) \defi \frac{\norm{x}\norm{y}}{\abs{x^{\Htran}B^{\prime}(\lambda)y}}
        \quad \text{and}\quad 
        \kvec(B,\lambda)\defi \bignorm{(Z_{\perp}^{\Htran}B(\lambda)Y_{\perp})^{-1}}.
    \end{equation*} 
    Here, we use $y^{\Htran}\Delta y= 0$ as the normalization condition for $y+\Delta y$, and 
    \begin{equation*}
        C= 5M\kval(B,\lambda)\bigl(\kval(B,\lambda)+\kvec(B,\lambda)\bigr)\max\bigl\{\kval(B,\lambda),\kvec(B,\lambda)\bigr\}.
    \end{equation*}
\end{theorem}

The proof of \cref{lem:pert} is provided in \cref{app:pfpert}. 
We refer to the quantities $\kval(B,\lambda)$ and $\kvec(B,\lambda)$ as the eigenvalue and eigenvector condition numbers, respectively.
These names are consistent with the standard eigenvalue problem setting \cite[pp.~48--50]{Stewart2001}.
Compared with existing first-order perturbation results, such as asymptotic and non-asymptotic bounds for linear eigenvalue problems in \cite{Stewart1990}, and asymptotic results for (homogeneous) polynomial eigenvalue problems in \cite{Dedieu2000}, our result is slightly different.
Strictly speaking, \cref{lem:pert} is a non-asymptotic perturbation bound.
However, unlike~\cite{Stewart1990}, where the non-asymptotic control applies to the first-order term, our explicit bound controls the second-order remainder.
This formulation preserves the sharpness of the first-order expansion while simultaneously providing a non-asymptotic guarantee on higher-order terms, whose importance will be further explained in \cref{rmk:nonasymptotic}.

\section{Convergence analysis}
\label{sec:theory}
This section studies the convergence of \cref{alg:RRR}.
We remark that our results are also new for standard eigenvalue problems $A(\xi)=A_{0}-\xi I$.

\subsection{Exact recovery}
Let $(\lambda,v)$ be a simple eigenpair of $A(\cdot)$. Suppose that $v$ is contained in a subspace $\mathcal{W}$. In the following lemma, we show that this eigenpair can be recovered almost surely by solving the compressed nonlinear eigenvalue problem of 
\begin{equation}
    \label{eq:defB}
    B(\xi) \defi \Omega^{\Htran} A(\xi)W, 
    \quad \text{where}\quad  W=[v,W_{\perp}]
\end{equation}
forms an orthonormal basis of $\mathcal{W}$. 
\begin{lemma}
    \label{lem:exact}
    Let $(\lambda,v)$ be a simple eigenpair of $A(\cdot)$.  Given an arbitrary fixed full-rank matrix $W\in\C^{n\times m}$ satisfying $We_{1}=v$, the following statements hold for almost every $\Omega\in\C^{n\times m}$ with respect to Lebesgue measure:
    The matrix-valued function $B(\xi)=\Omega^{\Htran}A(\xi)W$ is regular, and $(\lambda,e_{1})$ is a simple eigenpair of $B(\cdot)$.
\end{lemma}
\begin{proof}
    Since $A(\cdot)$ is regular, we know that $A(\xi_{0})$ is nonsingular for some $\xi_{0}\in\region$. Then $A(\xi_{0})W$ has full rank; hence, $\det B(\xi_{0}) = \det \bigl(\Omega^{\Htran}A(\xi_{0})W\bigr)\neq 0$ generically because this is a polynomial in the entries of $\Omega$.

Note that $B(\lambda)e_{1}=\Omega^{\Htran} A(\lambda)v=0$; hence, $(\lambda,e_{1})$ is an eigenpair of $B(\cdot)$.
We now show its simplicity. Since $\lambda$ is a simple eigenvalue of $A(\cdot)$, the null space of $A(\lambda)$ is $\spa{v}$. As $W$ has full rank and $We_{1}=v$, we have $\rk(A(\lambda)W)=m-1$. Then for almost every $\Omega$, we have $\rk B(\lambda) = \rk(\Omega^{\Htran}A(\lambda)W) =m-1$. Thus the null space of $B(\lambda)$ is $\spa{e_{1}}$. It remains to show that $x^{\Htran}B^{\prime}(\lambda)e_{1}\neq 0$ for a left eigenvector $x$ of $B(\cdot)$ corresponding to $\lambda$.

Write $W=[v,W_{\perp}]$. Then $\rk(A(\lambda)W_{\perp})=m-1$. Let $u$ be a left eigenvector of $A(\cdot)$ corresponding to the eigenvalue $\lambda$. We have $u^{\Htran}A(\lambda)W_{\perp}=0$ and $u^{\Htran}A^{\prime}(\lambda)v\neq 0$. Hence, $A^{\prime}(\lambda)v\notin\range(A(\lambda)W_{\perp})$, which implies that $\rk[A^{\prime}(\lambda)v,A(\lambda)W_{\perp}]=m$. Consequently, for a generic $\Omega$, the matrix $\Omega^{\Htran}[A^{\prime}(\lambda)v,A(\lambda)W_{\perp}]\in\C^{m\times m}$ is nonsingular.
Let $x$ be a left eigenvector of $B(\cdot)$ corresponding to $\lambda$. Since $x^{\Htran}B(\lambda)=x^{\Htran}\Omega^{\Htran}A(\lambda)W=0$, we have
\begin{equation}
    \label{eq:pfexact}
x^{\Htran}\Omega^{\Htran}[A^{\prime}(\lambda)v,A(\lambda)W_{\perp}]
=[x^{\Htran}B^{\prime}(\lambda)e_{1},0_{1\times(m-1)}].
\end{equation}
Since $\Omega^{\Htran}[A^{\prime}(\lambda)v,A(\lambda)W_{\perp}]$ is nonsingular and $x\neq 0$, we conclude that $x^{\Htran}B^{\prime}(\lambda)e_{1}\neq 0$.
\end{proof}

The proof of \cref{lem:exact} provides useful intuition for the subsequent analysis of robust recovery.
In fact, we can directly derive an upper bound for the eigenvalue condition number $\kval(B,\lambda)$ from \cref{eq:pfexact}:
\begin{equation*}
    \kval(B,\lambda)
    = \frac{\norm{x}\norm{e_{1}}}{\abs{x^{\Htran}B^{\prime}(\lambda)e_{1}}}
    = \frac{\norm{x}}{\norm{x^{\Htran}\Omega^{\Htran}[A^{\prime}(\lambda)v,A(\lambda)W_{\perp}]}}
    \leq \bignorm{\bigl(\Omega^{\Htran}[A^{\prime}(\lambda)v,A(\lambda)W_{\perp}]\bigr)^{-1}}.
\end{equation*}
Although this bound is not as sharp as what we will establish in \cref{lem:kappa}, it suggests studying the matrix $[A^{\prime}(\lambda)v,A(\lambda)W_{\perp}]$ and exploiting the orthogonality between $\Omega x$ and $A(\lambda)W_{\perp}$.

\subsection{Stability analysis}
Consider the ideal proxy $B(\cdot)$ defined in \cref{eq:defB}.
In this subsection, we derive probabilistic bounds for its eigenvalue condition number $\kval(B,\lambda)$ and eigenvector condition number $\kvec(B,\lambda)$, as introduced in \cref{lem:pert}. To this end, we first recall the following tail bounds for extreme singular values of complex Gaussian random matrices. 
The first result can be derived from the tail bound for the largest singular value of a real Gaussian random matrix \cite{davidson2001local} (see \cref{app:Gaussian}), while the second result can be found in \cite{Edelman1988}.
\begin{lemma}
    \label{lem:Gaussian}
    Let $\Omega\in\C^{n\times m}$ be a complex Gaussian random matrix. Then 
    \begin{equation*}
        \Pr\bigl(\norm{\Omega}\geq \sqrt{n}+\sqrt{m}+\sqrt{\ln(1/\delta)}\bigr)\leq \delta
        \quad\text{for all}\quad 0<\delta\leq 1.
    \end{equation*}
    When $m=n$, that is, $\Omega\in\C^{m\times m}$ is a square matrix (or a scalar), it holds that
    \begin{equation*}
        \Pr\bigl(\norm{\Omega^{-1}}\geq \sqrt{m/\delta}\bigr)\leq \delta \quad\text{for all}\quad \delta>0. 
    \end{equation*}
\end{lemma}

The following lemma establishes the stability of the RRR procedure: compressing the matrix $A(\lambda)W$ with a complex Gaussian random matrix $\Omega$ preserves both eigenvalue and eigenvector condition numbers with high probability.
We remark that replacing $\Omega$ with $W$, as in the standard Rayleigh--Ritz procedure, may completely destroy a well-conditioned eigenpair of $A(\cdot)$.

\begin{lemma}
    \label{lem:kappa}
    Let $u$ and $v$ be unit left and right eigenvectors of $A(\cdot)$ corresponding to a simple eigenvalue $\lambda$.
    Given an arbitrary fixed orthonormal matrix $W$ of the form $W=[v,W_{\perp}]\in\C^{n\times m}$, we define $B(\xi)=\Omega^{\Htran}A(\xi)W$, where $\Omega\in\C^{n\times m}$ is a complex Gaussian random matrix. Then 
    \begin{equation*}
        \begin{aligned}
                \kval(B,\lambda) &\leq  \frac{1}{\sqrt{\delta}}\cdot \kval(A,\lambda) &\text{holds with probability at least $1-\delta$,}\\ 
                \kvec(B,\lambda) &\leq  \frac{\sqrt{m-1}}{\sqrt{\delta}}\cdot \kvec(A,\lambda) &\text{holds with probability at least $1-\delta$.}\\ 
        \end{aligned}
    \end{equation*}
    Here, $\kval(A,\lambda)$ and $\kvec(A,\lambda)$ are the eigenvalue and eigenvector condition numbers of $A(\cdot)$:
    \begin{equation}
        \label{eq:defkappaA}
        \kval(A,\lambda)\defi \frac{\norm{u}\norm{v}}{\abs{u^{\Htran}A^{\prime}(\lambda)v}}
        \quad\text{and}\quad 
        \kvec(A,\lambda)\defi \bignorm{(Q_{\perp}^{\Htran}A(\lambda)V_{\perp})^{-1}},
    \end{equation}
    where $Q_{\perp}$ and $V_{\perp}$ are orthonormal bases of $\spa{A^{\prime}(\lambda)v}^{\perp}$ and $\spa{v}^{\perp}$, respectively.

\end{lemma}

\begin{proof}
    By \cref{lem:exact}, we know that $(\lambda,e_{1})$ is a simple eigenpair of $B(\cdot)$ almost surely. We further denote its corresponding unit left eigenvector by $x$.
    
    \emph{Eigenvalue condition number:}
    Since $\rk(A(\lambda)W_{\perp})=m-1$, we can let $P\in\C^{n\times (m-1)}$ be an orthonormal basis of $\range(A(\lambda)W_{\perp})$ and $p_{\perp} = (I-PP^{\Htran})A^{\prime}(\lambda)v$. Since $u^{\Htran}A(\lambda)=0$, we know that $u^{\Htran}P=0$ and 
    \begin{equation*}
        \norm{p_{\perp}} \geq \norm{uu^{\Htran}(I-PP^{\Htran})A^{\prime}(\lambda)v}= \norm{uu^{\Htran}A^{\prime}(\lambda)v} =  \abs{u^{\Htran}A^{\prime}(\lambda)v} = \frac{1}{\kval(A,\lambda)}.
    \end{equation*}
    Since $\rk B(\lambda)=m-1$, we know that
    \begin{equation*}
        \spa{x} = \range\bigl(B(\lambda)\bigr)^{\perp}=\range\bigl(\Omega^{\Htran}A(\lambda)W_{\perp}\bigr)^{\perp}=\range(\Omega^{\Htran}P)^{\perp},    
    \end{equation*}
    and in turn, $x^{\Htran}\Omega^{\Htran} p_{\perp} = x^{\Htran}\Omega^{\Htran} A^{\prime}(\lambda)v-x^{\Htran}\Omega^{\Htran} PP^{\Htran}A^{\prime}(\lambda)v = x^{\Htran}\Omega^{\Htran} A^{\prime}(\lambda)v$.
    Note that $\Omega^{\Htran}P$ and $\Omega^{\Htran}p_{\perp}$ are independent; consequently, $x$ and $\Omega^{\Htran}p_{\perp}$ are also independent. Thus, 
    \begin{equation*}
        \frac{1}{\kval(B,\lambda)\cdot\norm{p_{\perp}}} = \frac{\abs{x^{\Htran}B^{\prime}(\lambda)e_{1}}}{\norm{p_{\perp}}} = \frac{\abs{x^{\Htran}\Omega^{\Htran}A^{\prime}(\lambda)v}}{\norm{p_{\perp}}} = \frac{\abs{x^{\Htran}\Omega^{\Htran}p_{\perp}}}{\norm{p_{\perp}}} = \Bigabs{x^{\Htran}\Omega^{\Htran}\bigl(p_{\perp}/\norm{p_{\perp}}\bigr)}
    \end{equation*}
    is the magnitude of a complex Gaussian random variable. Using the tail bound in \cref{lem:Gaussian}, we know that, with probability at least $1-\delta$, it holds that 
    \begin{equation*}
        \kval(B,\lambda)\leq  \frac{1}{\sqrt{\delta}\cdot\norm{p_{\perp}}}  \leq \frac{1}{\sqrt{\delta}}\cdot\kval(A,\lambda).
    \end{equation*}

    \emph{Eigenvector condition number:} Let 
    \begin{equation*}
        Q_{\perp}Q_{\perp}^{\Htran}A(\lambda)W_{\perp} = QR
    \end{equation*}
    be the thin QR decomposition. Since the matrix $[A^{\prime}(\lambda)v,A(\lambda)V_{\perp}]$ is nonsingular and $\range(W_{\perp})\subset \range(V_{\perp})$, we know that $Q\in\C^{n\times (m-1)}$ and
    \begin{equation*}
        \norm{R^{-1}} \leq \bignorm{\bigl(Q_{\perp}^{\Htran}A(\lambda)V_{\perp}\bigr)^{-1}} = \kvec(A,\lambda),
    \end{equation*}
    where we use the interlacing property on the smallest singular values of $Q_{\perp}^{\Htran}A(\lambda)W_{\perp}$ and $Q_{\perp}^{\Htran}A(\lambda)V_{\perp}$. 
    Recall that 
    \begin{equation*}
        \kvec(B,\lambda) =  \bignorm{\bigl((\Omega Z_{\perp})^{\Htran}A(\lambda)W_{\perp}\bigr)^{-1}}, 
    \end{equation*}
    where $Z_{\perp}$ is an orthonormal basis of $\spa{\Omega^{\Htran}A^{\prime}(\lambda)v}^{\perp}$. We know that $(\Omega Z_{\perp})^{\Htran}A^{\prime}(\lambda)v=0$; hence, $\Omega Z_{\perp} = Q_{\perp}Q_{\perp}^{\Htran}\Omega Z_{\perp}$. Thus, 
    \begin{equation*}
        \begin{aligned}
            \kvec(B,\lambda) &= 
        \bignorm{\bigl((\Omega Z_{\perp})^{\Htran}Q_{\perp}Q_{\perp}^{\Htran}A(\lambda)W_{\perp}\bigr)^{-1}} = \bignorm{\bigl((\Omega Z_{\perp})^{\Htran}QR\bigr)^{-1}} \\ 
        &\leq \bignorm{\bigl(Z_{\perp}^{\Htran}(\Omega^{\Htran}Q)\bigr)^{-1}}\cdot \norm{R^{-1}}\leq \bignorm{\bigl(Z_{\perp}^{\Htran}(\Omega^{\Htran}Q)\bigr)^{-1}}\cdot \kvec(A,\lambda).        
        \end{aligned}
    \end{equation*}
    Note that $A^{\prime}(\lambda)v$ and $Q$ are orthogonal. We know that $\Omega^{\Htran}A^{\prime}(\lambda)v$ and $\Omega^{\Htran}Q$ are independent; consequently, $Z_{\perp}$ and $\Omega^{\Htran}Q$ are also independent. Then $Z_{\perp}^{\Htran}(\Omega^{\Htran}Q)\in\C^{(m-1)\times (m-1)}$ is a complex Gaussian random matrix. By \cref{lem:Gaussian},  
    \begin{equation*}
        \bignorm{\bigl(Z_{\perp}^{\Htran}(\Omega^{\Htran}Q)\bigr)^{-1}}\leq \sqrt{(m-1)/\delta}
    \end{equation*} 
    holds with probability at least $1-\delta$. 
\end{proof}

\subsection{Main result: convergence of randomized Ritz pairs}
\begin{theorem}
\label{thm:ritzpair}
Let $(\lambda,v)$ be a simple eigenpair of $A(\cdot)$.
For $0<\delta<1/4$, define
\begin{equation}
    \label{eq:defCr}
\Crval \defi 
\frac{C_{0}}{\sqrt{\delta}}
\cdot \kval(A,\lambda),
\quad\text{and}\quad 
\Crvec \defi  
1 + \frac{C_{0}\sqrt{m-1}}{\sqrt{\delta}}
\cdot \kvec(A,\lambda),
\end{equation}
where $C_{0}=\bigl(\sqrt{m}+1+\sqrt{\ln(1/\delta)}\bigr)\norm{A(\lambda)}$, and $\kval(A,\lambda)$ and $\kvec(A,\lambda)$ are defined in~\cref{eq:defkappaA}.
Then there exist constants $\epsilon_{0}>0$ and $C>0$, depending only on $m$, $n$, $\delta$, $\kval(A,\lambda)$, $\kvec(A,\lambda)$, $\norm{A}_{\region}$, $\norm{A^{\prime}}_{\region}$, $\norm{A^{\prime\prime}}_{\region}$, $\mathrm{dist}(\lambda,\partial\region)$,
such that the following statement holds.

Consider an arbitrary fixed orthonormal matrix $W_{\epsilon}\in\C^{n\times m}$ with $\epsilon=\angle(v,W_{\epsilon})\leq \epsilon_{0}$.
Let $B_{\epsilon}(\xi)=\Omega^{\Htran}A(\xi)W_{\epsilon}$, where $\Omega\in\C^{n\times m}$ is a complex Gaussian random matrix. Then, with probability at least $1-4\delta$, there exists an eigenpair $(\mu_{\epsilon},y_{\epsilon})$ of $B_{\epsilon}(\cdot)$ such that
\begin{equation}
    \label{eq:convRRR}
    \abs{\mu_{\epsilon}-\lambda} \leq \Crval\epsilon + C\epsilon^{2}
    \quad\text{and}\quad
    \angle(v,W_{\epsilon}y_{\epsilon}) \leq \Crvec\epsilon+ C\epsilon^{2}.
\end{equation}
\end{theorem}

\begin{remark}
    \label{rmk:nonasymptotic}
To achieve linear convergence, the coefficients $\{\Crval,\Crvec,C\}$ must be independent of the choice of $W_{\epsilon}$. Moreover, the failure event depends on $W_{\epsilon}$. Consequently, \cref{lem:pert}, on which \cref{thm:ritzpair} relies, cannot be replaced by asymptotic results, since we \textbf{cannot} take a union bound over a \textbf{family} of bases $W_{\epsilon}$.
\end{remark}

\begin{proof}[Proof of \cref{thm:ritzpair}]
    Since the choice of orthonormal basis of $\range(W_{\epsilon})$ does not affect randomized Ritz pairs, without loss of generality, we assume
    \begin{equation*}
        W_{\epsilon} = 
        \begin{bmatrix}
            v\cos\epsilon+w\sin\epsilon & W_{\perp}
        \end{bmatrix},
        \quad\text{where}\quad 
        \begin{bmatrix}
            v,w,W_{\perp}
        \end{bmatrix}
        \quad\text{is orthonormal.}
    \end{equation*}
    Let $B(\xi)=\Omega^{\Htran}A(\xi)W$, where $W=[v,W_{\perp}]$.
    Taking a union bound over \cref{lem:Gaussian,lem:exact,lem:kappa}, we know that, with probability at least $1-4\delta$, the following statements hold:
    \begin{equation}
        \label{eq:Fval}
        \begin{aligned}
            &\text{$B(\cdot)$ is regular and $(\lambda,e_{1})$ is its simple eigenpair},\\ 
            &\norm{\Omega}\leq \sqrt{n}+\sqrt{m}+\sqrt{\ln(1/\delta)},\\ 
            &\bignorm{\Omega^{\Htran}A(\lambda)[v,w]}\leq C_{0},\\ 
            & \kval(B,\lambda) \leq  \frac{1}{\sqrt{\delta}}\cdot \kval(A,\lambda),\\ 
            & \kvec(B,\lambda) \leq \frac{\sqrt{m-1}}{\sqrt{\delta}}\cdot \kvec(A,\lambda).
        \end{aligned}
    \end{equation}
    
    In what follows, we assume that \cref{eq:Fval} holds and apply \cref{lem:pert} to $B(\cdot)$ and $B_{\epsilon}(\cdot)$.
    By submultiplicativity, we know that, for any integer $k\geq 0$ and $\xi\in\region$,
    \begin{equation*}
        \Bigabs{\frac{\de^{k}}{\de \xi^{k}} e_{i}^{\Htran}B(\xi)e_{j}} = \Bigabs{(\Omega e_{i})^{\Htran} \frac{\de^{k}}{\de \xi^{k}} A(\xi ) (We_{j})} \leq \norm{\Omega} \bignorm{ \frac{\de^{k}}{\de \xi^{k}} A(\xi )},
    \end{equation*}
    implying that $B(\cdot)$ satisfies \cref{asp:A}. In particular, 
    \begin{equation*}
        \max\bigl\{\norm{B}_{\region},\norm{B^{\prime}}_{\region},\norm{B^{\prime\prime}}_{\region}\bigr\}\leq \bigl(\sqrt{n}+\sqrt{m}+\sqrt{\ln(1/\delta)}\bigr)\cdot 
        \max\bigl\{\norm{A}_{\region},\norm{A^{\prime}}_{\region},\norm{A^{\prime\prime}}_{\region}\bigr\}.
    \end{equation*}
    Let $\Delta B(\xi) \defi B_{\epsilon}(\xi)-B(\xi) =\Omega^{\Htran}A(\xi)(W_{\epsilon}-W)$. Note that
    \begin{equation}
        \label{eq:WeW}
        W_{\epsilon}-W  = \begin{bmatrix}
            v(\cos\epsilon-1)+w\sin\epsilon & 0_{n\times (m-1)}
        \end{bmatrix}
        = \begin{bmatrix}
            \begin{bmatrix}
                v & w
            \end{bmatrix}
            \begin{bmatrix}
                \cos\epsilon-1\\ 
                \sin\epsilon
            \end{bmatrix}& 0_{n\times (m-1)}
        \end{bmatrix}.
    \end{equation}
    Since  $(\cos\epsilon-1)^{2}+\sin^{2}\epsilon=4\sin^{2}(\epsilon/2)\leq \epsilon^{2}$, we know that $\norm{W_{\epsilon}-W}\leq \epsilon$ and hence
    \begin{equation*}
        \max\bigl\{\norm{\Delta B}_{\region},\norm{\Delta B^{\prime}}_{\region},\norm{\Delta B^{\prime\prime}}_{\region}\bigr\}\leq \bigl(\sqrt{n}+\sqrt{m}+\sqrt{\ln(1/\delta)}\bigr)\cdot 
        \max\bigl\{\norm{A}_{\region},\norm{A^{\prime}}_{\region},\norm{A^{\prime\prime}}_{\region}\bigr\}\cdot \epsilon.
    \end{equation*}
    When $\epsilon_{0}$ is sufficiently small, \cref{lem:pert} yields the existence of an eigenpair $(\mu_{\epsilon},y_{\epsilon})$ of $B_{\epsilon}(\cdot)$ satisfying
    \begin{equation}
        \label{eq:pfpair}
        \begin{aligned}
            \abs{\mu_{\epsilon}-\lambda} &\leq \kval(B,\lambda)\cdot\norm{\Delta B(\lambda)} +C\epsilon^{2}\leq \frac{\Crval}{C_{0}}\norm{\Delta B(\lambda)} +C\epsilon^{2},\\ 
            \tan\angle(e_{1},y_{\epsilon})&\leq \kvec(B,\lambda)\cdot\norm{\Delta B(\lambda)} +C\epsilon^{2} \leq \frac{\Crvec-1}{C_{0}}\norm{\Delta B(\lambda)} +C\epsilon^{2},
        \end{aligned}
    \end{equation}
    where $C$ is a constant depending only on $m$, $n$, $\delta$, $\kval(A,\lambda)$, $\kvec(A,\lambda)$, $\norm{A}_{\region}$, $\norm{A^{\prime}}_{\region}$, $\norm{A^{\prime\prime}}_{\region}$.
    According to \cref{eq:WeW}, it holds that 
    \begin{equation*}
        \norm{\Delta B(\lambda)}  
        =\Bignorm{\Omega^{\Htran}A(\lambda)\begin{bmatrix}
            v & w
        \end{bmatrix}
        \begin{bmatrix}
            \cos\epsilon-1\\ 
            \sin\epsilon
        \end{bmatrix}}\leq \Bignorm{\Omega^{\Htran}A(\lambda)\begin{bmatrix}
            v & w
        \end{bmatrix}} \cdot\epsilon \leq C_{0}\epsilon,
    \end{equation*}
    where we use \cref{eq:Fval} in the second inequality.
    Combining this with \cref{eq:pfpair}, we obtain the convergence of randomized Ritz values in \cref{eq:convRRR} and 
    \begin{equation*}
        \tan\angle(e_{1},y_{\epsilon})\leq (\Crvec-1)\epsilon +C\epsilon^{2}.
    \end{equation*}
    Then the convergence of randomized Ritz vectors in \cref{eq:convRRR} is proved by the triangle inequality:
    \begin{equation*}
        \angle(v,W_{\epsilon}y_{\epsilon})\leq 
        \angle(v,W_{\epsilon}e_{1})+\angle(W_{\epsilon}e_{1},W_{\epsilon}y_{\epsilon}) \leq \epsilon+\tan\angle(e_{1},y_{\epsilon}) \leq  \Crvec\epsilon+C\epsilon^{2},
    \end{equation*}
    where we use $\spa{W_{\epsilon}e_{1}}=\spa{W_{\epsilon}W_{\epsilon}^{\Htran}v}$ in the second inequality.
\end{proof}

\subsection{Refining randomized Ritz values}
\label{sec:refine}
This subsection investigates the accuracy of refined randomized Ritz values obtained from the strategies described in \Cref{sec:algo}.
With the convergence of randomized Ritz vectors in \cref{thm:ritzpair}, the convergence of refined randomized Ritz values is actually straightforward. 
Before presenting results for general situations, we first state the result for standard eigenvalue problems.

\begin{theorem}
Let $(\lambda,v)$ be a simple eigenpair of a matrix $A_{0}\in\C^{n\times n}$. For $0<\delta<1/4$, we define $\Crvec$, $C$ and $\epsilon_{0}$ as in \cref{thm:ritzpair}. 
Consider an arbitrary fixed orthonormal matrix $W_{\epsilon}\in\C^{n\times m}$ with $\epsilon=\angle(v,W_{\epsilon})\leq \epsilon_{0}$.
Let $\Omega\in\C^{n\times m}$ be a complex Gaussian random matrix. With probability at least $1-4\delta$, there exists an eigenpair $(\mu_{\epsilon},y_{\epsilon})$ of the matrix pencil $\Omega^{\Htran}A_{0}W_{\epsilon}-\xi \Omega^{\Htran}W_{\epsilon}$ such that
\begin{equation*}
    \abs{\rho_{\epsilon} - \lambda}
    \leq \norm{A_{0} - \lambda I}\,\sin^{\gamma} \bigl(\Crvec \epsilon+C\epsilon^{2}\bigr),\quad \text{where}\quad \rho_{\epsilon} = \frac{(W_{\epsilon}y_{\epsilon})^{\Htran}A_{0}(W_{\epsilon}y_{\epsilon})}{(W_{\epsilon}y_{\epsilon})^{\Htran}(W_{\epsilon}y_{\epsilon})},
\end{equation*}
where $\gamma = 2$ if $A_{0}$ is Hermitian and $\gamma = 1$ otherwise.

\end{theorem}

The proof follows by substituting the convergence result for randomized Ritz vectors in \cref{thm:ritzpair} into the standard analysis of the Rayleigh quotient.  
Compared with randomized Ritz values $\mu_{\epsilon}$, the refinement $\rho_{\epsilon}$ leverages the Hermiticity of $A_{0}$ to achieve quadratic convergence, analogous to the standard Rayleigh--Ritz procedure for exterior eigenvalues.  
Moreover, when $v$ is (nearly) a left eigenvector of $A_{0}$ corresponding to $\lambda$, the accuracy of $\rho_{\epsilon}$ is substantially higher than that of $\mu_{\epsilon}$.

\paragraph{Refinement via Rayleigh functional.} 
\!We adopt the following lemma from \cite[Cor.~18 and Thm.~21]{Schwetlick2012}, which establishes both the existence and the accuracy of the Rayleigh functional for a general (not necessarily Hermitian) matrix-valued function.
\begin{lemma}
    Let $(\lambda,v)$ be a simple eigenpair of $A(\cdot)$. 
    If $v^{\Htran}A^{\prime}(\lambda)v\neq 0$, then there exists a unique Rayleigh functional $\Rq(\cdot)$ defined in a neighborhood of $v$. Moreover, for a vector $w_{\epsilon}$ with $\epsilon\defi\angle(v,w_{\epsilon})$, when $\epsilon$ is sufficiently small, it holds that 
    \begin{equation*}
        \abs{\Rq(w_{\epsilon}) - \lambda}
        \leq C_{\gamma}
        \frac{\norm{A(\lambda)}\norm{v}^{2}}{\abs{v^{\Htran}A^{\prime}(\lambda)v}}\tan^{\gamma}\epsilon,
    \end{equation*}
    where $\gamma = 2$ and $C_{\gamma} = 8/3$ if the matrix $A(\lambda)$ is Hermitian, while $\gamma = 1$ and $C_{\gamma} = 10/3$ otherwise.
\end{lemma}

Combining this lemma with our convergence result for randomized Ritz vectors in \cref{thm:ritzpair}, we obtain the following convergence result for refined randomized Ritz values.

\begin{theorem}
    \label{thm:Rf}
    Let $(\lambda,v)$ be a simple eigenpair of $A(\cdot)$.
    Assume that $v^{\Htran}A^{\prime}(\lambda)v\neq 0$. 
    For $0<\delta<1/4$, we define $\Crvec$, $C$ and $\epsilon_{0}$ as in \cref{thm:ritzpair}. 
    Consider an arbitrary fixed orthonormal matrix $W_{\epsilon}\in\C^{n\times m}$ with $\epsilon=\angle(v,W_{\epsilon})\leq \epsilon_{0}$.
    Let $B_{\epsilon}(\xi)=\Omega^{\Htran}A(\xi)W_{\epsilon}$, where $\Omega\in\C^{n\times m}$ is a complex Gaussian random matrix. With probability at least $1-4\delta$, there exists an eigenpair $(\mu_{\epsilon},y_{\epsilon})$ of $B_{\epsilon}(\cdot)$ such that
    \begin{equation*}
        \abs{\Rq(W_{\epsilon}y_{\epsilon})-\lambda}\leq  
        C_{\gamma}\frac{\norm{A(\lambda)}\norm{v}^{2}}{\abs{v^{\Htran}A^{\prime}(\lambda)v}}
        \tan^{\gamma} \bigl(\Crvec \epsilon+C\epsilon^{2}\bigr),
    \end{equation*}
    where $\gamma = 2$ and $C_{\gamma} = 8/3$ if the matrix $A(\lambda)$ is Hermitian, while $\gamma = 1$ and $C_{\gamma} = 10/3$ otherwise.
\end{theorem}

\paragraph{Refinement via stationary points.}
As noted in \Cref{sec:algo}, a Rayleigh functional may not exist for certain matrix-valued functions. In such cases, one can instead refine randomized Ritz values by finding stationary points of \cref{eq:nls}. The following lemma demonstrates that the convergence of randomized Ritz vectors implies the convergence of the corresponding refinement.

\begin{lemma}
    \label{lem:ift}
    Let $(\lambda,v)$ be a simple eigenpair of $A(\cdot)$. Given a vector $w_{\epsilon}$ with $\epsilon\defi\angle(v,w_{\epsilon})$, when $\epsilon$ is sufficiently small, there exists a stationary point 
    \begin{equation*}
        \rho_{\epsilon}\in \Bigl\{ \rho\in\C \Bigm\vert \frac{\de }{\de \bar{\rho}} \norm{A(\rho)w_{\epsilon}}^{2} = 0  \Bigr\},
    \end{equation*}
    and a constant $C_{A}$ depending only on $A(\cdot)$, such that 
    \begin{equation*}
        \abs{\rho_{\epsilon}-\lambda} \leq  \frac{\norm{A(\lambda)}\norm{v}}{\norm{A^{\prime}(\lambda)v}}\epsilon+C_{A}\epsilon^{2}.
    \end{equation*}
\end{lemma}

The proof is based on the implicit function theorem and Taylor expansion; we defer it to \cref{app:ift}. By plugging the convergence of randomized Ritz vectors established in \cref{thm:ritzpair} into \cref{lem:ift}, we obtain the convergence of refined randomized Ritz values obtained by finding stationary points of \cref{eq:nls}.

\begin{theorem}
    \label{thm:nls}
    Let $(\lambda,v)$ be a simple eigenpair of $A(\cdot)$.
    For $0<\delta<1/4$, we define $\Crvec$, $C$ and $\epsilon_{0}$ as in \cref{thm:ritzpair}. 
    Consider an arbitrary fixed orthonormal matrix $W_{\epsilon}\in\C^{n\times m}$ with $\epsilon=\angle(v,W_{\epsilon})\leq \epsilon_{0}$.
    Let $B_{\epsilon}(\xi)=\Omega^{\Htran}A(\xi)W_{\epsilon}$, where $\Omega\in\C^{n\times m}$ is a complex Gaussian random matrix. With probability at least $1-4\delta$, there exists an eigenpair $(\mu_{\epsilon},y_{\epsilon})$ of $B_{\epsilon}(\cdot)$ such that
    \begin{equation*}
        \abs{\rho_{\epsilon}-\lambda}\leq  
        \frac{\norm{A(\lambda)}\norm{v}}{\norm{A^{\prime}(\lambda)v}}
        \Crvec \epsilon+C_{A}(\Crvec^{2}+C)\epsilon^{2}
    \end{equation*}
    holds for some stationary point $\rho_{\epsilon}$ of \cref{eq:nls}, where $C_{A}$ is a constant depending only on $A(\cdot)$.
\end{theorem}

\section{Numerical experiments}
\label{sec:numexp}
In this section, we present several numerical experiments.
Let $(\lambda,v)$ denote a target eigenpair. The accuracy of an approximate eigenpair $(\mu,w)$ is measured by
\begin{equation}
    \label{exp:err}
    \abs{\lambda-\mu} \quad\text{and}\quad \angle(v,w),
\end{equation}
respectively.
The experiments are designed to examine different aspects of the RRR procedure. First, we consider neutral modes of indefinite pencils, where the standard Rayleigh--Ritz procedure fails while the RRR procedure remains robust. Second, we solve the \texttt{butterfly} nonlinear eigenvalue problem from NLEVP \cite{Betcke2013} and illustrate that when $A(\lambda)$ is Hermitian, refined randomized Ritz values exhibit quadratic convergence comparable to that of (standard) Ritz values. Finally, we investigate the empirical failure probability of the RRR procedure, illustrating the sharpness of the factor $1/\sqrt{\delta}$ in \cref{thm:ritzpair}. 
This experiment also indicates that, for the tested example, this factor is not improved by oversampling.

\subsection{Neutral modes of indefinite pencils}
Consider a matrix pencil $A_{0}-\xi A_{1}\in\C^{2n\times 2n}$ with the Hermitian indefinite matrix
\begin{equation*}
    A_{1}=\begin{bmatrix}
    0 & I\\
    I & 0
\end{bmatrix}.
\end{equation*}
We are interested in neutral modes \cite{Bognar1974}, that is, eigenvectors $v$ of the pencil $A_{0}-\xi A_{1}$ satisfying $v^{\Htran} A_{1} v = 0$.
In this experiment, we consider a smoothly parameterized $A_{0}(\tau)$ constructed as follows. Let 
\begin{equation*}
    Q = \frac{1}{2}\begin{bmatrix}
        Q_{1}+Q_{2} & Q_{1}-Q_{2}\\
        Q_{1}-Q_{2} & Q_{1}+Q_{2}
    \end{bmatrix}\in\C^{2n\times 2n},
\end{equation*}
where $Q_{1},Q_{2}\in\C^{n\times n}$ are random unitary matrices, each obtained as the $Q$-factor of the QR decomposition of an $n\times n$ complex Gaussian random matrix.
Then $Q$ is also a unitary matrix and $Q^{\Htran}A_{1}Q=A_{1}$.
Given a unit vector $v_{1}\in\C^{n}$, we define
\begin{equation*}
    v_{1}(\tau) = \eta(\tau)\exp\bigl((\tau-\tau_{0}) G\bigr)v_{1}
    \quad\text{and}\quad
    v(\tau) = Q\begin{bmatrix}
        v_{1}(\tau)\\ 0
    \end{bmatrix},
\end{equation*}
where $G\in\C^{n\times n}$ is a complex Gaussian random matrix.
Here $\eta(\tau)$ is a normalization factor such that $\norm{v(\tau)}=1$. Writing $P(\tau)=v_{1}(\tau)v_{1}(\tau)^{\Htran}$, we define the parameter-dependent matrix
\begin{equation}
    \label{eq:defHam}
    A_{0}(\tau) = Q\begin{bmatrix}
        (I-P(\tau))(G_{11}+G_{11}^{\Htran})(I-P(\tau)) &-\overline{\lambda} P(\tau)-(I-P(\tau))G_{21}^{\Htran} \\
        \lambda P(\tau)+G_{21}(I-P(\tau)) &G_{22}+G_{22}^{\Htran}
    \end{bmatrix}Q^{\Htran},
\end{equation}
where $G_{11},G_{22}\in\C^{n\times n}$ are also complex Gaussian random matrices, $G_{21}$ will be specified later, and $\lambda$ is the target eigenvalue.
The matrices $Q$, $G$, $G_{11}$, $G_{21}$, $G_{22}$ and the scalar $\lambda$ are fixed, so that $A_{0}(\cdot)$ is smooth in $\tau$.
Then by construction, we know that $A_{0}(\tau)v(\tau)=\lambda A_{1}v(\tau)$ and $v(\tau)^{\Htran}A_{1}v(\tau)=0$ for every $\tau$; that is, $v(\tau)$ is a neutral mode of $A_{0}(\tau)-\xi A_{1}$ associated with the eigenvalue $\lambda$ for every $\tau$. Moreover, since $A_{1}Q=QA_{1}$, we know that
\begin{equation}
    \label{eq:iso}
    v(\tau)^{\Htran}A_{1}v(\tau^{\prime})=0
    \quad\text{for all} \quad \tau,\tau^{\prime}.
\end{equation}

Suppose that we have computed $v(\tau_{k})$ for some integers $k$, and our goal is to approximate $v(\tau_{0})$ and $\lambda$. In the following, we omit the parameter when it equals $\tau_{0}$. Let $\mathcal{W}_{k}=\spa{v(\tau_{1}),\dotsc,v(\tau_{k})}$. Then $\angle(v,\mathcal{W}_{k})$ is usually very small for a smooth $A_{0}(\cdot)$. However, by \cref{eq:iso}, we have $W_{k}^{\Htran}A_{1}W_{k}=0$. Consequently, the compressed pencil $W_{k}^{\Htran}A_{0}W_{k}-\xi W_{k}^{\Htran}A_{1}W_{k}$ reduces to the constant matrix $W_{k}^{\Htran}A_{0}W_{k}$, which is generically nonsingular. Hence, in exact arithmetic, the compressed pencil has no finite eigenvalues and the standard Rayleigh--Ritz procedure breaks down. In the following, we show that RRR works well for this task.

Setting $n=2000$ and $\lambda=1$, we choose $v_{1}$ as a normalized complex Gaussian random vector. We consider two choices of $G_{21}$: a zero matrix and a complex Gaussian random matrix.
When $G_{21}=0$, one has $v^{\Htran}(A_{1}A_{0}-\lambda I)w=0$ for all $w\in\mathcal{W}_{k}$, so the first-order term of $\rho-\lambda$ vanishes for every $w\in\mathcal{W}_{k}$, where $\rho$ is the refinement in \cref{eq:nls}. In this case, we expect quadratic convergence of refined randomized Ritz values.

Let $\tau_{k}=10^{-3}\cdot k$. We report numerical results for $k=1,\dotsc,10$ in \cref{fig:ham}. It is clear that (standard) Ritz values cannot provide a meaningful approximation of the target eigenvalue. In exact arithmetic, (standard) Ritz vectors are not well-defined. Due to roundoff errors, the standard Rayleigh--Ritz procedure returns Ritz vectors, but they lose up to three digits of accuracy.
The RRR procedure works well in this situation: the accuracy of randomized Ritz vectors is nearly optimal, that is, they closely follow the subspace angles $\angle(v,\mathcal{W}_{k})$; and refined randomized Ritz values achieve quadratic convergence when $G_{21}=0$ and linear convergence for a complex Gaussian $G_{21}$.

\begin{figure}[htbp]
    \centering
    \includegraphics[width=\figsizeD]{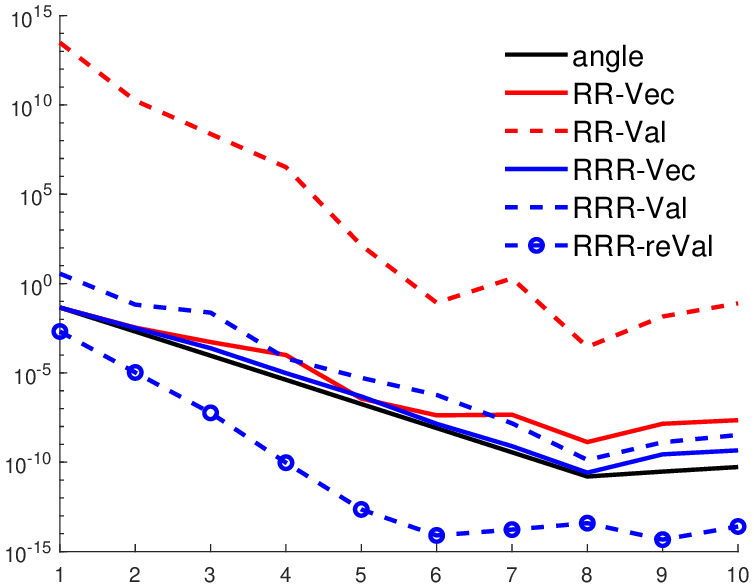}
    \hspace{1cm}
    \includegraphics[width=\figsizeD]{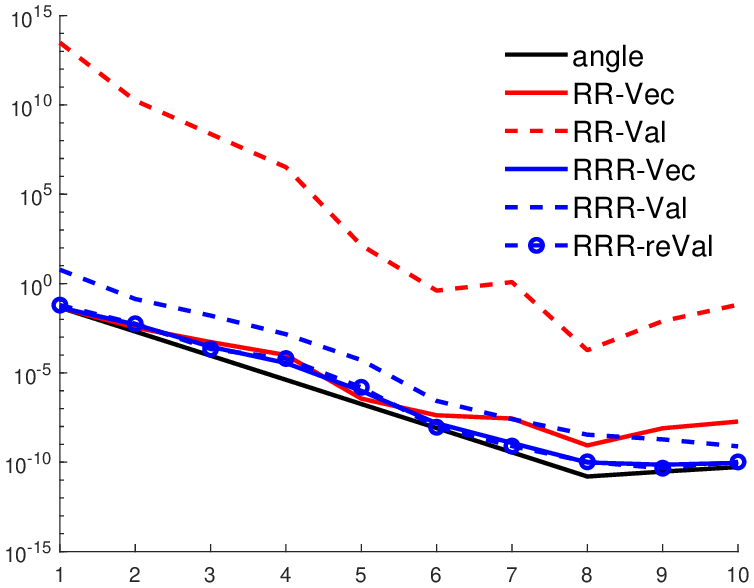}
    \caption{Convergence of approximate neutral modes extracted from $\mathcal{W}_k$. The X-axis shows the dimension $k$ of $\mathcal{W}_{k}$, while the Y-axis reports the accuracy of approximate eigenpairs measured by \cref{exp:err}. Here, ``angle'' denotes $\angle(v,\mathcal{W}_k)$; ``RR-Vec/Val'' denote Ritz vectors/values from the standard Rayleigh--Ritz procedure; ``RRR-Vec/Val/reVal'' denote randomized Ritz vectors/values and refined randomized Ritz values from the RRR procedure (\cref{alg:RRR}). The matrix $G_{21}$ in \cref{eq:defHam} is zero in the left panel and a complex Gaussian random matrix in the right panel.} 
    \label{fig:ham}
\end{figure}

\subsection{The \texttt{butterfly} nonlinear eigenvalue problem}
Let  
\begin{equation*}
    A(\xi) = \sum_{i=0}^{4}c_{i}\xi^{i}A_{i},\quad\text{where}\quad A_{i}\in\R^{n\times n},
\end{equation*}
the coefficients $c_{i}$ are real, $A_{0}$, $A_{2}$ and $A_{4}$ are symmetric, and $A_{1}$ and $A_{3}$ are skew-symmetric.
Using the default parameters $c_{i}$ from \cite{Betcke2013} and $n=4096$, we compute the eigenvalue closest to $\sigma\in\{2\mi,1+1\mi\}$ with a nonlinear Arnoldi-type method \cite{Mehrmann2004}.
Specifically, we take $w_{0}$ as a complex Gaussian random starting vector, and then perform the residual inverse iteration \cite{Neumaier1985}: 
\begin{equation*}
    w_{k+1} = \eta_{k+1}\Bigl(w_{k}-\bigl(A(\sigma)\bigr)^{-1}A(\rho_{k})w_{k}\Bigr),
\end{equation*}
where $\eta_{k+1}$ is a normalization factor such that $\norm{w_{k+1}}=1$, and 
\begin{equation*}
    \rho_{k}=\Rq(w_{k}) = \argmin_{\rho\in\C} \bigl\{\abs{\rho-\sigma}\mid w_{k}^{\Htran}A(\rho)w_{k}=0 \bigr\}.
\end{equation*}
We construct a trial subspace as $\mathcal{W}_{k}=\spa{w_{1},\dotsc,w_{k}}$, and perform the standard and randomized Rayleigh--Ritz procedures to extract eigenpair approximations. The reference solution is computed by performing the residual inverse iteration until the residual norm is below $10^{-12}$.
Since the RRR procedure depends on the random sketch, we repeat it with $1000$ independent sketches for each $\mathcal{W}_{k}$ and report the median and the $10\%$ and $90\%$ quantiles of the resulting errors.
Numerical results are reported in \cref{fig:butterfly}. 
When $\sigma=2\mi$, the target eigenvalue $\lambda$ is purely imaginary, which makes $A(\lambda)$ Hermitian, and both (standard) Ritz values and refined randomized Ritz values converge quadratically. In this situation, the behavior of the standard and randomized Rayleigh--Ritz procedures is similar. In the other situation, since $A(\lambda)$ is non-Hermitian, both methods exhibit linear convergence, and (refined) randomized Ritz values and vectors achieve one to two additional digits of accuracy compared to (standard) Ritz values and vectors. Moreover, the narrow spread between the $10\%$ and $90\%$ quantiles shows that this behavior is consistent across sketches rather than an artifact of a particular realization.

\begin{figure}[htbp]
    \centering
    \includegraphics[width=\figsizeD]{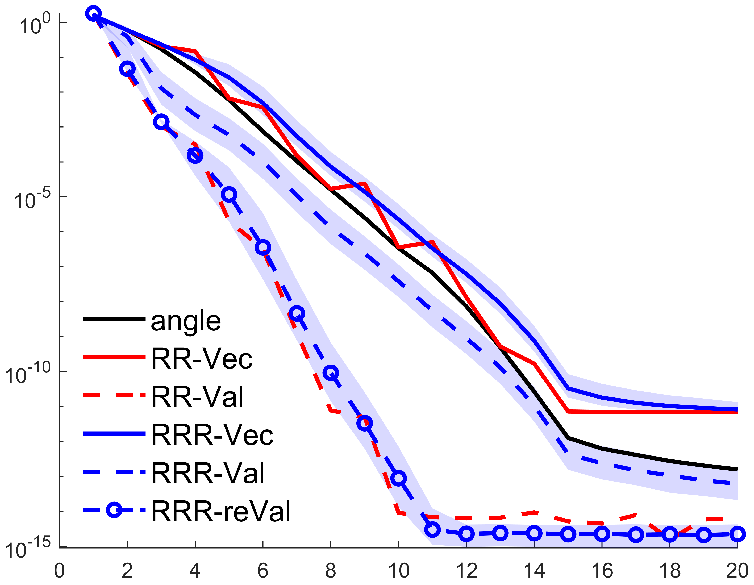}
    \hspace{1cm}
    \includegraphics[width=\figsizeD]{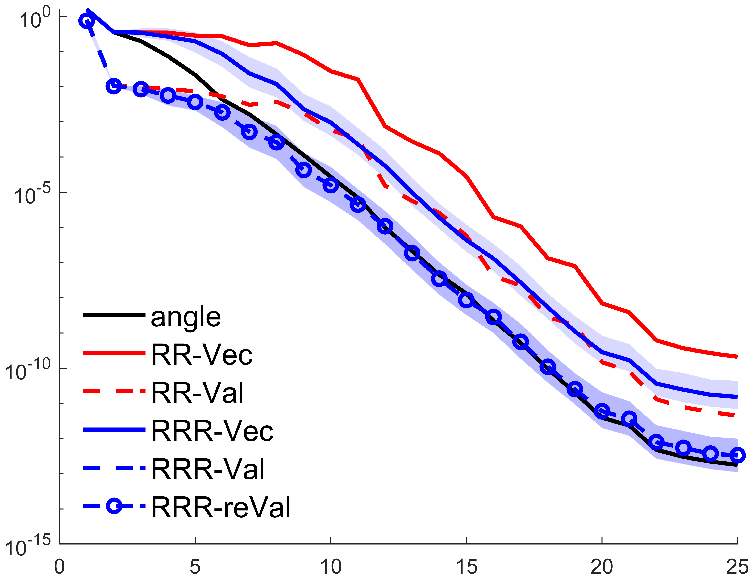}
    \caption{Convergence of approximate eigenpairs of the \texttt{butterfly} eigenvalue problem extracted from $\mathcal{W}_{k}$. The X-axis shows the dimension $k$ of $\mathcal{W}_{k}$, while the Y-axis reports the accuracy of approximate eigenpairs measured by \cref{exp:err}. Here, ``angle'' denotes $\angle(v,\mathcal{W}_k)$; ``RR-Vec/Val'' denote Ritz vectors/values from the standard Rayleigh--Ritz procedure; ``RRR-Vec/Val/reVal'' denote randomized Ritz vectors/values and refined randomized Ritz values from the RRR procedure (\cref{alg:RRR}). For the RRR procedure, lines show the median over $1000$ independent sketches, and shaded regions show the range between the $10\%$ and $90\%$ quantiles. In the left panel, the curves for RR-Val and RRR-reVal nearly coincide; in the right panel, the curves for RRR-Val and RRR-reVal are almost indistinguishable, so that RRR-Val is hidden behind RRR-reVal.
    The matrix $A(\lambda)$ is Hermitian for the left panel, and non-Hermitian for the right panel.}
    \label{fig:butterfly}
\end{figure}

\subsection{Empirical probability and oversampling}

\label{sec:numexpFP}

The prefactors $\Crval$ and $\Crvec$ in \cref{thm:ritzpair} both depend on $1/\sqrt{\delta}$. In certain randomized algorithms, such as the randomized SVD \cite{Halko2011}, this issue can be mitigated by oversampling; that is, by replacing the $n\times m$ complex Gaussian matrix $\Omega$ with an $n\times (m+s)$ complex Gaussian matrix $\Omega_{\mathrm{os}}$ for some integer $s \geq 1$.  
However, we observe that such a strategy is not effective in our setting. 

We illustrate this with a numerical experiment. Consider a matrix pencil $A_{0}-\xi A_{1}$, where $A_{0},A_{1}\in\C^{1000\times 1000}$ are both complex Gaussian random matrices. We are interested in the eigenvalue closest to $0.01$. Let $W_{\epsilon}\in\C^{1000\times 10}$ be an orthonormal basis computed by running $10$ steps of subspace shift-and-invert iterations from a complex Gaussian random starting block, where $\epsilon\approx 6.4\times 10^{-8}$. Then we extract approximate eigenpairs by the RRR procedure with and without oversampling. For oversampling, we use complex Gaussian matrices $\Omega_{\mathrm{os}}\in\C^{1000\times 20}$, that is, $m=s=10$, and solve the rectangular eigenvalue problem of the pencil $\Omega_{\mathrm{os}}^{\Htran}A_{0}W_{\epsilon}-\xi \Omega_{\mathrm{os}}^{\Htran}A_{1}W_{\epsilon}\in\C^{20\times 10}$ by the minimal perturbation approach proposed in \cite{Ito2016}, which essentially solves a total least squares problem via SVD. With $2^{17}$ independent trials on $\Omega$ and $\Omega_{\mathrm{os}}$, we collect the empirical results in \cref{fig:fp}, which indicates that the factor $1/\sqrt{\delta}$ is sharp and, for this example, such a factor cannot be improved by oversampling through the minimal perturbation rectangular formulation.

\begin{figure}[htbp]
    \centering
    \includegraphics[width=\figsizeD]{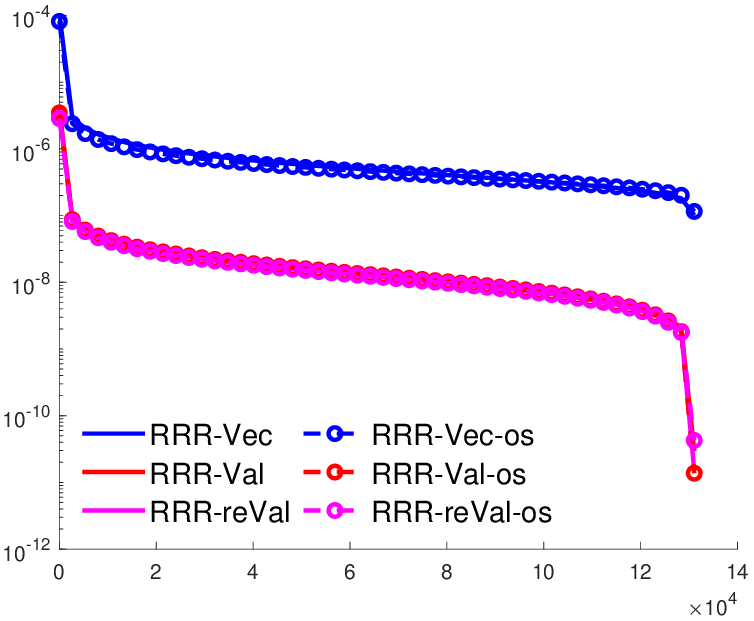}
    \hspace{1cm}
    \includegraphics[width=\figsizeD]{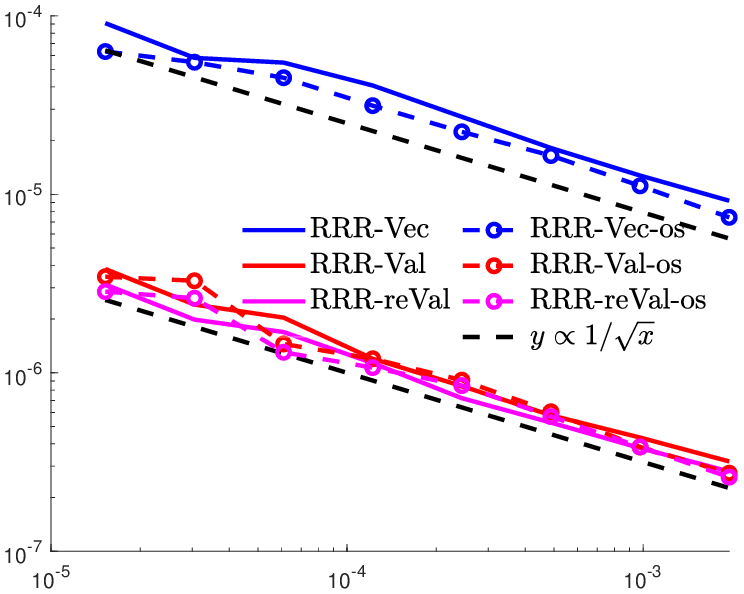}
    \caption{Empirical results from $2^{17}$ samples on the accuracy (measured by \cref{exp:err}) of the RRR procedure (\cref{alg:RRR}).
    Here, ``RRR-Vec/Val/reVal'' denote randomized Ritz vectors/values and refined randomized Ritz values, and ``os'' denotes oversampling.
    The right panel reports the accuracy versus frequency.}
    \label{fig:fp}
\end{figure}

\section{Conclusion}
In this paper, we have introduced a randomized Rayleigh--Ritz (RRR) procedure for eigenpair extraction.
For a simple eigenpair, the RRR procedure is stable with high probability, without requiring any further assumption on the compressed eigenvalue problem, in contrast to the standard Rayleigh--Ritz procedure.
Future work includes exploring ways to preserve Hermiticity for Hermitian problems and investigating how randomization can stabilize subspace methods in other contexts, such as evaluating matrix functions applied to vectors.

\section*{Acknowledgments}
The author thanks Daniel Kressner for helpful discussions and constructive feedback that significantly improved this work.

\section*{Declarations}
\begin{itemize}
    \item AI disclosure: The author discloses that, besides grammar checking, \cref{app:pfpert} was drafted and checked, following the approach of an earlier version by the author, and the related literature was searched, with the assistance of Claude Opus 5.5 (Anthropic); all arguments were verified by the author, who takes full responsibility for the content.
\item Code availability: Scripts for reproducing the numerical results are publicly available at \url{https://github.com/nShao678/randomized-Rayleigh--Ritz}.
\end{itemize}

\appendix
\section{Proof of \cref{lem:pert}}
\label{app:pfpert}
The existence part of \cref{lem:pert} follows from the Newton--Kantorovich theorem in the following form \cite[p.~536]{KantorovichAkilov1982}; see also \cite[Thm.~A.1]{KarowKressner2014}.

\begin{theorem}[Newton--Kantorovich]
    \label{thm:NK}
    Let $\mathcal{X}$ and $\mathcal{Y}$ be Banach spaces and let $F\colon\mathcal{P}\to\mathcal{Y}$ be twice continuously differentiable on an open set $\mathcal{P}\subset\mathcal{X}$ containing $p_{0}$. Suppose that there is a linear operator $T\colon\mathcal{X}\to\mathcal{Y}$ with continuous inverse such that
    \begin{equation*}
        \norm{T^{-1}F(p_{0})}\leq\eta_{0},\qquad
        \norm{T^{-1}F^{\prime}(p_{0})-I}\leq\eta_{1},\qquad
        \norm{T^{-1}F^{\prime\prime}(p)}\leq\eta_{2}\quad\text{for all } p\in\mathcal{P}.
    \end{equation*}
    If $\eta_{1}<1$, $\eta\defi\eta_{0}\eta_{2}/(1-\eta_{1})^{2}<1/2$, and $\mathcal{P}$ contains the closed ball of radius $r_{0}\defi 2\eta_{0}/\bigl((1-\eta_{1})(1+\sqrt{1-2\eta})\bigr)$ around $p_{0}$, then $F$ has a zero $p_{*}\in\mathcal{P}$ with $\norm{p_{*}-p_{0}}\leq r_{0}$.
\end{theorem}

\begin{proof}[Proof of \cref{lem:pert}]
Without loss of generality, we assume $\norm{x}=\norm{y}=1$, and write $\kappa_{1}=\kval(B,\lambda)=1/\abs{x^{\Htran}B^{\prime}(\lambda)y}$ and $\kappa_{2}=\kvec(B,\lambda)$.

\emph{Linearization:}
Let $J_{0}=[B^{\prime}(\lambda)y,B(\lambda)Y_{\perp}]\in\C^{m\times m}$. By arguments similar to those in the proof of \cref{lem:exact}, we know that $J_{0}$ is nonsingular. Using $x^{\Htran}B(\lambda)=0$ and $Z_{\perp}^{\Htran}B^{\prime}(\lambda)y=0$, we obtain
\begin{equation*}
    \begin{bmatrix} x^{\Htran}\\ Z_{\perp}^{\Htran}\end{bmatrix}J_{0}
    =\begin{bmatrix} x^{\Htran}B^{\prime}(\lambda)y & 0\\ 0 & Z_{\perp}^{\Htran}B(\lambda)Y_{\perp}\end{bmatrix}.
\end{equation*}
Inverting this identity yields, for any $g\in\C^{m}$,
\begin{equation*}
    J_{0}^{-1}g=\begin{bmatrix}
        x^{\Htran}g/\bigl(x^{\Htran}B^{\prime}(\lambda)y\bigr)\\
        (Z_{\perp}^{\Htran}B(\lambda)Y_{\perp})^{-1}Z_{\perp}^{\Htran}g
    \end{bmatrix}.
\end{equation*}
Since $\norm{x}=1$ and $Z_{\perp}$ has orthonormal columns, we have  
\begin{equation*}
    \bigabs{x^{\Htran}g/\bigl(x^{\Htran}B^{\prime}(\lambda)y\bigr)}\leq \kappa_{1}\norm{g}
    \quad\text{and}\quad
    \norm{(Z_{\perp}^{\Htran}B(\lambda)Y_{\perp})^{-1}Z_{\perp}^{\Htran}g}\leq \kappa_{2}\norm{g}.
\end{equation*}
Let us consider $\mathcal{X}=\C\times\C^{m-1}$ with the weighted norm 
\begin{equation*}
    \norm{(\alpha,h)}_{\ast}\defi\max\{\abs{\alpha}/\kappa_{1},\norm{h}/\kappa_{2}\}
\end{equation*}
and $\mathcal{Y}=\C^{m}$ with the Euclidean norm. Defining $T\colon\mathcal{X}\to\mathcal{Y}$ by $
T(\alpha,h)=J_{0}\begin{bmatrix}
    \alpha\\
    h
\end{bmatrix}$,
we thus have
\begin{equation}
    \label{eq:appTinv}
    \norm{T^{-1}g}_{\ast}\leq\norm{g}\qquad\text{for all } g\in\C^{m}.
\end{equation}
Since $\kappa_{1}\geq 1/\norm{B^{\prime}(\lambda)}\geq 1/M$, the first condition in~\cref{eq:pertRadius} implies $(\kappa_{1}+\kappa_{2})\epsilon\leq 1/16$ and, as $M\kappa_{1}(\kappa_{1}+\kappa_{2})\geq M\kappa_{1}^{2}\geq 1/M$, also $\epsilon\leq M/16$.

\emph{Existence:}
An eigenpair of $B+\Delta B$ of the form $(\mu,y+Y_{\perp}z)$, which satisfies the normalization $y^{\Htran}(Y_{\perp}z)=0$, corresponds to a zero of the vector-valued function
\begin{equation*}
    F(\mu,z)\defi\bigl(B(\mu)+\Delta B(\mu)\bigr)\,(y+Y_{\perp}z).
\end{equation*}
We apply \cref{thm:NK} with $p_{0}=(\lambda,0)$ and $\mathcal{P}=\{p\in\mathcal{X}\colon\norm{p-p_{0}}_{\ast}<2\epsilon\}$.
For $(\mu,z)\in\mathcal{P}$, the second condition in~\cref{eq:pertRadius} gives $\abs{\mu-\lambda}<2\kappa_{1}\epsilon<\mathrm{dist}(\lambda,\partial\region)$, so $F$ is well defined and smooth on $\mathcal{P}$, and $\norm{z}<2\kappa_{2}\epsilon\leq 1$. 
By definition, we have, for $(\mu,z)\in\mathcal{P}$ and $(\alpha,h),(\widetilde{\alpha},\widetilde{h})\in\mathcal{X}$,
\begin{equation*}
    \begin{aligned}
        F^{\prime}(\mu,z)(\alpha,h)&=\alpha\bigl(B^{\prime}(\mu)+\Delta B^{\prime}(\mu)\bigr)(y+Y_{\perp}z)+\bigl(B(\mu)+\Delta B(\mu)\bigr)Y_{\perp}h,\\
        F^{\prime\prime}(\mu,z)\bigl[(\alpha,h),(\widetilde{\alpha},\widetilde{h})\bigr]&=\alpha\widetilde{\alpha}\bigl(B^{\prime\prime}(\mu)+\Delta B^{\prime\prime}(\mu)\bigr)(y+Y_{\perp}z)+\bigl(B^{\prime}(\mu)+\Delta B^{\prime}(\mu)\bigr)Y_{\perp}(\alpha\widetilde{h}+\widetilde{\alpha}h).
    \end{aligned}
\end{equation*}
Using~\cref{eq:appTinv} and $B(\lambda)y=0$, we have that  
\begin{equation*}
    \begin{aligned}
        \norm{T^{-1}F(p_{0})}_{\ast}&\leq\norm{F(p_{0})}=\norm{\Delta B(\lambda)y}\leq\norm{\Delta B(\lambda)},\\
        \bignorm{\bigl(T^{-1}F^{\prime}(p_{0})-I\bigr)(\alpha,h)}_{\ast}&\leq\norm{\alpha\Delta B^{\prime}(\lambda)y+\Delta B(\lambda)Y_{\perp}h}\\
        &\leq\epsilon\bigl(\abs{\alpha}+\norm{h}\bigr)\leq(\kappa_{1}+\kappa_{2})\epsilon\norm{(\alpha,h)}_{\ast},\\
        \bignorm{T^{-1}F^{\prime\prime}(\mu,z)\bigl[(\alpha,h),(\widetilde{\alpha},\widetilde{h})\bigr]}_{\ast}&\leq(M+\epsilon)\bigl((1+\norm{z})\abs{\alpha}\abs{\widetilde{\alpha}}+\abs{\alpha}\norm{\widetilde{h}}+\abs{\widetilde{\alpha}}\norm{h}\bigr)\\
        &\leq 4M\kappa_{1}(\kappa_{1}+\kappa_{2})\norm{(\alpha,h)}_{\ast}\norm{(\widetilde{\alpha},\widetilde{h})}_{\ast},
    \end{aligned}
\end{equation*}
where $(\mu,z)\in\mathcal{P}$ is arbitrary, and the last inequality uses $\norm{z}\leq 1$ and $\epsilon\leq M$.
Thus, the assumptions of \cref{thm:NK} hold with
\begin{equation*}
    \eta_{0}=\norm{\Delta B(\lambda)},\qquad
    \eta_{1}=(\kappa_{1}+\kappa_{2})\epsilon,\qquad
    \eta_{2}=4M\kappa_{1}(\kappa_{1}+\kappa_{2}).
\end{equation*}
By~\cref{eq:pertRadius}, $\eta_{1}\leq 1/16$ and $\eta_{0}\eta_{2}\leq 4M\kappa_{1}(\kappa_{1}+\kappa_{2})\epsilon\leq 1/4$, so that $\eta\leq\frac{1}{4}(\frac{16}{15})^{2}=\frac{64}{225}<\frac{1}{2}$ and $(1-\eta_{1})(1+\sqrt{1-2\eta})\geq\frac{15}{16}(1+\frac{\sqrt{97}}{15})>\frac{3}{2}$. Therefore, $r_{0}\leq 4\eta_{0}/3\leq 4\epsilon/3<2\epsilon$, i.e., the closed ball of radius $r_{0}$ around $p_{0}$ is contained in $\mathcal{P}$. Hence \cref{thm:NK} yields a zero $p_{*}=(\lambda+\Delta\lambda,z)$ of $F$ with $\norm{p_{*}-p_{0}}_{\ast}\leq4\epsilon/3$.

\emph{First-order bounds:}
Let $d=p_{*}-p_{0}=(\Delta\lambda,z)$ and $r=\norm{d}_{\ast}\leq4\epsilon/3$. Taylor's formula on the segment $[p_{0},p_{*}]\subset\mathcal{P}$ gives
\begin{equation*}
    0=F(p_{*})=F(p_{0})+F^{\prime}(p_{0})d+R
    \quad\text{with}\quad
    R=\int_{0}^{1}(1-t)F^{\prime\prime}(p_{0}+td)[d,d]\de t.
\end{equation*}
Applying $T^{-1}$ and using the bounds above, we obtain
\begin{equation*}
    r\leq\norm{\Delta B(\lambda)}+\eta_{1}r+\eta_{2}r^{2}/2
    \leq\norm{\Delta B(\lambda)}+5M\kappa_{1}(\kappa_{1}+\kappa_{2})\epsilon^{2},
\end{equation*}
where we used $r\leq4\epsilon/3$ and $M\kappa_{1}\geq 1$. Since $\abs{\Delta\lambda}\leq\kappa_{1}r$ and $\norm{\Delta y}=\norm{z}\leq\kappa_{2}r$ by the definition of $\norm{\cdot}_{\ast}$, the proof is complete with $C=5M\kappa_{1}(\kappa_{1}+\kappa_{2})\max\{\kappa_{1},\kappa_{2}\}$.
\end{proof}

\section{Proof of \cref{lem:Gaussian}}
\label{app:Gaussian}
Without loss of generality, we assume that $m\leq n$.
In the following, $\chi_{k}$ denotes the distribution of the square root of a $\chi^{2}$ random variable with $k$ degrees of freedom. 
We first show that there exist unitary matrices $U\in\C^{n\times n}$ and $V\in\C^{m\times m}$ such that
\begin{equation*}
    U^{\Htran}\Omega V = \begin{bmatrix}
        B\\ 0_{(n-m)\times m}
    \end{bmatrix},
    \qquad
    B = \begin{bmatrix}
        \alpha_{1} & \beta_{1} & & \\
        & \alpha_{2} & \ddots & \\
        & & \ddots & \beta_{m-1}\\
        & & & \alpha_{m}
    \end{bmatrix}\in\R^{m\times m},
\end{equation*}
where $\alpha_{1},\dotsc,\alpha_{m},\beta_{1},\dotsc,\beta_{m-1}$ are independent,
\begin{equation*}
    \sqrt{2}\alpha_{i}\sim \chi_{2n-2(i-1)}
    \quad\text{and}\quad \sqrt{2}\beta_{i}\sim \chi_{2m-2i}.
\end{equation*}
Let $\omega_{1}$ be the first column of $\Omega$ and let $U_{1}\in\C^{n\times n}$ be a unitary matrix, depending only on $\omega_{1}$, such that $U_{1}^{\Htran}\omega_{1}=\norm{\omega_{1}}e_{1}$. Since $U_{1}$ is independent of the remaining columns of $\Omega$, unitary invariance gives
\begin{equation*}
    U_{1}^{\Htran}\Omega=\begin{bmatrix}
        \alpha_{1} & w^{\Htran}\\
        0 & \Omega^{\prime}
    \end{bmatrix},
    \qquad \alpha_{1}=\norm{\omega_{1}},
\end{equation*}
where $\sqrt{2}\alpha_{1}\sim\chi_{2n}$, and $w\in\C^{m-1}$ and $\Omega^{\prime}\in\C^{(n-1)\times(m-1)}$ are complex Gaussian and independent of each other and of $\alpha_{1}$. Next, let $V_{1}\in\C^{(m-1)\times(m-1)}$ be a unitary matrix, depending only on $w$, such that $w^{\Htran}V_{1}=\norm{w}e_{1}^{\Ttran}$. Then
\begin{equation*}
    U_{1}^{\Htran}\Omega\begin{bmatrix}
        1 & \\
        & V_{1}
    \end{bmatrix}=\begin{bmatrix}
        \alpha_{1} & \beta_{1}e_{1}^{\Ttran}\\
        0 & \Omega_{1}
    \end{bmatrix},
    \qquad \beta_{1}=\norm{w},
\end{equation*}
where $\sqrt{2}\beta_{1}\sim\chi_{2(m-1)}$, and $\Omega_{1}\in\C^{(n-1)\times(m-1)}$ is complex Gaussian and independent of $\alpha_{1}$ and $\beta_{1}$. We then repeat these two steps on $\Omega_{1}$; the corresponding transformations act only on rows $2,\dotsc,n$ and columns $3,\dotsc,m$, and hence leave $\alpha_{1}$ and $\beta_{1}$ unchanged. After $m$ steps, we obtain the claimed decomposition, where $U$ and $V$ accumulate the unitary transformations applied from the left and from the right, respectively.

Applying the same procedure with orthogonal transformations to a real Gaussian random matrix $\widetilde{\Omega}\in\R^{2n\times 2m}$, we obtain orthogonal matrices $\widetilde{U}\in\R^{2n\times 2n}$ and $\widetilde{V}\in\R^{2m\times 2m}$ such that $\widetilde{U}^{\Ttran}\widetilde{\Omega}\widetilde{V}$ is upper bidiagonal with independent entries
\begin{equation*}
    \widetilde{\alpha}_{i}\sim \chi_{2n-(i-1)}
    \quad\text{and}\quad
    \widetilde{\beta}_{i}\sim \chi_{2m-i}
\end{equation*}
on the diagonal and superdiagonal, respectively. Let $\widetilde{B}\in\R^{m\times m}$ denote its leading $m\times m$ principal submatrix. Since the spectral norm of a submatrix does not exceed that of the whole matrix, we have
\begin{equation*}
    \bignorm{\widetilde{\Omega}}=\bignorm{\widetilde{U}^{\Ttran}\widetilde{\Omega}\widetilde{V}}\geq \bignorm{\widetilde{B}}.
\end{equation*}
Since, for each $1\leq i\leq m$, $(\alpha_{i},\beta_{i})$ and $(\widetilde{\alpha}_{i},\widetilde{\beta}_{i})$ are both pairs of $\chi$ random variables and the degrees of freedom of the latter are higher, we know that there exists a nonnegative random matrix $\Delta B$ such that $\widetilde{B}$ and $\sqrt{2}B+\Delta B$ have the same distribution. Consequently, for every $s\geq0$,
\begin{equation*}
    \Pr\bigl(\norm{\Omega}\geq s\bigr)
    =\Pr\bigl(\norm{B}\geq s\bigr)
    \leq\Pr\bigl(\bignorm{\widetilde{B}}\geq \sqrt{2}s\bigr)
    \leq\Pr\bigl(\bignorm{\widetilde{\Omega}}\geq \sqrt{2}s\bigr),
\end{equation*}
where we used the fact that $B$ and $\widetilde{B}$ are both nonnegative matrices.
By the tail bound for the norm of real Gaussian random matrices \cite[Thm.~II.13]{davidson2001local}, we know that 
\begin{equation*}
    \Pr\Bigl(\bignorm{\widetilde{\Omega}}\geq\sqrt{2n}+\sqrt{2m}+\sqrt{2\ln(1/\delta)}\Bigr)\leq\delta
    \quad\text{for all}\quad 0<\delta\leq 1.
\end{equation*}
Combining the above two inequalities gives the desired tail bound for $\norm{\Omega}$.
\begin{remark}
    An alternative approach to obtaining a tail bound for $\norm{\Omega}$ is to use
    \begin{equation*}
        \norm{\Omega} \leq \norm{\Re(\Omega)}+\norm{\Im(\Omega)},
    \end{equation*}
    which only gives a cruder tail estimate around $\sqrt{2n}+\sqrt{2m}$ rather than  $\sqrt{n}+\sqrt{m}$.
\end{remark}
\section{Proof of \cref{lem:ift}}
\label{app:ift}

To avoid the proliferation of constants, we use $C_{A}$ to denote generic constants depending only on $A(\cdot)$, whose values may change from line to line. 

Let us first show that, as $\epsilon\to 0$, there exists a stationary point
\begin{equation}
    \label{eq:apprhoeps}
    \rho_{\epsilon} \in \Bigl\{ \rho\in\C \Bigm\vert \frac{\de }{\de \bar{\rho}} \norm{A(\rho)w_{\epsilon}}^{2} = 0  \Bigr\}, 
    \quad \text{such that}\quad 
    \abs{\rho_{\epsilon}-\lambda}\leq C_{A}\epsilon.
\end{equation}
Without loss of generality, we assume that $\norm{v}=\norm{w_{\epsilon}}=1$ and $\norm{v-w_{\epsilon}}=2\sin(\epsilon/2)\leq \epsilon$. Denote 
\begin{equation*}
    \lambda = \alpha_{0}+\beta_{0}\mi
    \quad\text{and}\quad 
    v = x_{0}+y_{0}\mi,
    \quad\text{where}\quad
    \alpha_{0},\beta_{0}\in\R
    \quad\text{and}\quad 
    x_{0},y_{0}\in\R^{n}.   
\end{equation*}
Consider the \emph{real} analytic function
\begin{equation*}
    f(\alpha,\beta,x,y) = \norm{A(\rho)w}^{2},
    \quad\text{where}\quad
    \rho=\alpha+\beta\mi 
    \quad\text{and}\quad 
    w = x+y\mi.
\end{equation*}
Since $A(\lambda)v=0$, we have $f(\alpha_{0},\beta_{0},x_{0},y_{0})=0$. Define the function $G\colon \R^{2n+2}\mapsto \R^{2}$ as 
\begin{equation*}
    G(\alpha,\beta,x,y) = \begin{bmatrix}
        \nabla_{\alpha} f(\alpha,\beta,x,y)\\ 
        \nabla_{\beta} f(\alpha,\beta,x,y)
    \end{bmatrix}.
\end{equation*}
We know $G(\alpha_{0},\beta_{0},x_{0},y_{0})=0$. Since $A(\cdot)$ is holomorphic, Taylor expansion yields that 
\begin{equation*}
    A(\lambda+\Delta \lambda) v = A(\lambda)v+\Delta \lambda\cdot A^{\prime}(\lambda) v +\order(\abs{\Delta \lambda}^{2}) = \Delta \lambda\cdot A^{\prime}(\lambda) v +\order(\abs{\Delta \lambda}^{2}).
\end{equation*}
Since $\lambda$ is a simple eigenvalue, we know $u^{\Htran}A^{\prime}(\lambda)v\neq 0$, implying that $A^{\prime}(\lambda)v\neq 0$. Thus, 
\begin{equation*}
    \norm{A(\lambda+\Delta \lambda)v}^{2} = \abs{\Delta \lambda}^{2}\norm{A^{\prime}(\lambda)v}^{2}+\order(\abs{\Delta \lambda}^{3}),
\end{equation*}
implying that the Jacobian matrix of $G$ is 
\begin{equation*}
    \nabla_{\alpha,\beta}^{2}f(\alpha_{0},\beta_{0},x_{0},y_{0}) = 2\norm{A^{\prime}(\lambda)v}^{2}I_{2}\succ 0.
\end{equation*}
The implicit function theorem asserts that there exist real analytic functions $f_{\alpha}(x,y)$ and $f_{\beta}(x,y)$ defined in a neighborhood of $(x_{0},y_{0})$ such that 
\begin{equation*}
    \alpha_{0}=f_{\alpha}(x_{0},y_{0}),\quad 
    \beta_{0}=f_{\beta}(x_{0},y_{0}),\quad 
    G\bigl(f_{\alpha}(x,y),f_{\beta}(x,y),x,y\bigr) = 0.
\end{equation*}
Here both $f_{\alpha}$ and $f_{\beta}$ are independent of $w_{\epsilon}$.
Since $\norm{w_{\epsilon}-v}\leq \epsilon$, we know that $\norm{\Re w_{\epsilon}-x_{0}}\leq \epsilon$ and $\norm{\Im w_{\epsilon}-y_{0}}\leq \epsilon$. When $\epsilon$ is sufficiently small, we can take 
\begin{equation*}
    \rho_{\epsilon} = f_{\alpha}(\Re w_{\epsilon},\Im w_{\epsilon})+f_{\beta}(\Re w_{\epsilon},\Im w_{\epsilon})\mi \in \Bigl\{ \rho\in\C \Bigm\vert \frac{\de }{\de \bar{\rho}} \norm{A(\rho)w_{\epsilon}}^{2} = 0  \Bigr\},
\end{equation*}
such that $\abs{\rho_{\epsilon}-\lambda}\leq C_{A}\epsilon$. Thus, \cref{eq:apprhoeps} is justified.

Then we further investigate $\abs{\rho_{\epsilon}-\lambda}$. By the stationarity condition in \cref{eq:apprhoeps}, we have 
\begin{equation}
    \label{eq:apppfT1}
    \bigl(A^{\prime}(\rho_{\epsilon})w_{\epsilon}\bigr)^{\Htran}\bigl(A(\rho_{\epsilon})w_{\epsilon}\bigr) = 0.
\end{equation}
Taylor expansions of $A(\rho_{\epsilon})w_{\epsilon}$ and $A^{\prime}(\rho_{\epsilon})w_{\epsilon}$ give that
\begin{equation*}
    A(\rho_{\epsilon})w_{\epsilon} = (\rho_{\epsilon}-\lambda)A^{\prime}(\lambda)v+A(\lambda)(w_{\epsilon}-v)+r_{2}, 
    \quad\text{and}\quad 
    A^{\prime}(\rho_{\epsilon})w_{\epsilon} = A^{\prime}(\lambda)v +r_{1},        
\end{equation*}
where $\norm{r_{2}}\leq C_{A}\epsilon^{2}$ and $\norm{r_{1}}\leq C_{A}\epsilon$. 
Here, we use $\norm{v-w_{\epsilon}}\leq \epsilon$, $\abs{\lambda-\rho_{\epsilon}}\leq C_{A}\epsilon$ and $A(\lambda)v=0$. 
Substituting these two expansions into \cref{eq:apppfT1}, we have 
\begin{equation*}
    \Bigabs{
    (\rho_{\epsilon}-\lambda)\norm{A^{\prime}(\lambda)v}^{2}+\bigl(A^{\prime}(\lambda)v\bigr)^{\Htran}\bigl(A(\lambda)(w_{\epsilon}-v)\bigr)} \leq C_{A}\epsilon^{2}.
\end{equation*}
Rearranging this inequality and using submultiplicativity, we complete the proof:
\begin{equation*}
    \abs{\rho_{\epsilon}-\lambda}\leq \frac{\norm{A(\lambda)}}{\norm{A^{\prime}(\lambda)v}}\epsilon+C_{A}\epsilon^{2}.
\end{equation*}

\end{document}